\documentclass[10pt,reqno]{amsart}
\usepackage[utf8]{inputenc}
\usepackage{amsmath}
\usepackage{amsfonts}
\usepackage{amssymb}
\usepackage{graphicx}
\usepackage[left=2.4cm,right=2.4cm,top=2.5cm,bottom=2.5cm]{geometry}
\numberwithin{equation}{section}
\graphicspath{ {images/} }
\usepackage{amsmath,amssymb, amsthm} 
\usepackage{lipsum}
\usepackage{stmaryrd}
\usepackage{dsfont}
\usepackage{comment}
\usepackage[all]{xy}

\makeatletter
\def\@settitle{\begin{center}%
  \baselineskip14\p@\relax
  \bfseries
  \uppercasenonmath\@title
  \@title
  \ifx\@subtitle\@empty\else
     \\[1ex]\uppercasenonmath\@subtitle
     \footnotesize\mdseries\@subtitle
  \fi
  \end{center}%
}
\def\subtitle#1{\gdef\@subtitle{#1}}
\def\@subtitle{}
\makeatother

\usepackage{tikz-cd}
\usepackage{hyperref}
\usepackage{bm}
\hypersetup{colorlinks, linkcolor=blue, citecolor=black, urlcolor=black}

\usepackage[english]{babel}
\usepackage[colorinlistoftodos]{todonotes}

\theoremstyle{plain}
\newtheorem{thm}{Theorem}[subsection] 

\usepackage{mathrsfs}

\theoremstyle{definition}

\newtheorem{rmk}[thm]{Remark}

\newtheorem{rem}[thm]{Remark}
\newtheorem*{rmk-intro}{Remark}
\theoremstyle{definition}

\theoremstyle{plain}
\newtheorem{prop}[thm]{Proposition}
\theoremstyle{plain}
\newtheorem{lemma}[thm]{Lemma}
\theoremstyle{plain}
\newtheorem{cor}[thm]{Corollary}
\theoremstyle{plain}
\newtheorem{conj}[thm]{Conjecture}
\theoremstyle{plain}
\newtheorem{thmintro}{Theorem}

\theoremstyle{plain}
\newtheorem{corintro}{Corollary}

\newtheorem{conjintro}{Conjecture}


\newcounter{parentnumber}

\newcommand{\Tac}{\mathbf{T}}
\newcommand{\Tc}{\mathbf{T}}

\usepackage{bbm}

\makeatletter  
\newcommand{\colim@}[2]{%
  \vtop{\m@th\ialign{##\cr
    \hfil$#1\operator@font colim$\hfil\cr
    \noalign{\nointerlineskip\kern1.5\ex@}#2\cr
    \noalign{\nointerlineskip\kern-\ex@}\cr}}%
}
\newcommand{\colim}{%
  \mathop{\mathpalette\colim@{\rightarrowfill@\scriptscriptstyle}}\nmlimits@
}
\renewcommand{\varprojlim}{%
  \mathop{\mathpalette\varlim@{\leftarrowfill@\scriptscriptstyle}}\nmlimits@
}
\renewcommand{\varinjlim}{%
  \mathop{\mathpalette\varlim@{\rightarrowfill@\scriptscriptstyle}}\nmlimits@
}
\makeatother

\usepackage{multicol}

\newcommand{\Z}{\mathbb{Z}}
\newcommand{\Q}{\mathbb{Q}}

\newcommand{\R}{\mathbb{R}}

\newcommand{\F}{\mathcal{F}}

\font\russ=wncyr10  1
\def\sha{\hbox{\russ\char88}}

\newcommand{\Lcal}{\mathcal{L}}
\newcommand{\Fcal}{\mathcal{F}}

\newcommand{\rH}{\mathrm{H}}

\begin{document}
\title{On refined nonvanishing conjectures by Kurihara and Kolyvagin}

\author{Francesc Castella}
\address[F.~Castella]{University of California Santa Barbara, South Hall, Santa Barbara, CA 93106, USA}
\email{castella@ucsb.edu}

\author{Takamichi Sano}
\address[T.~Sano]{Osaka Metropolitan University, Department of Mathematics, 3-3-138 Sugimoto, Sumiyoshi-ku, Osaka, 558-8585, Japan}
\email{tsano@omu.ac.jp}

\begin{abstract}
In this paper, we extend the results of \cite{BCGS} on refined conjectures by Kurihara and Kolyvagin, allowing primes of any reduction type in the case of Kurihara's conjectures, and inert primes in the underlying imaginary quadratic field in the case of Kolyvagin's. The key innovation is a new approach to the computation of the $p$-divisibility index of certain special elements in Galois cohomology (the bottom class of a $\Lambda$-adic Euler system twisted by a character sufficiently close to the trivial character) 
based on a reformulation of the Iwasawa Main Conjectures in terms of determinants of Selmer complexes.
\end{abstract}

\date{\today}
\maketitle
\tableofcontents

\section{Introduction}

In \cite{BCGS}, the first author with Burungale, Grossi and Skinner gave a proof of nonvanishing conjectures by Kolyvagin \cite{kolystructure} and Kurihara \cite{kurihara-iwasawa2012}, and their `quantitative' refinements \cite{zhang-CDM,kim}, for elliptic curves $E/\Q$ at certain primes $p$. Building on a new approach to the descent computations in \cite{BCGS}, in this paper we obtain an extension of the results in \emph{op.\,cit.} allowing $E$ to have any reduction type at $p$
in the case of the refined Kurihara conjectures, and $p$ to be any prime unramified   
in the underlying quadratic imaginary field $K$ in the case of the refined Kolyvagin conjectures. 

\subsection{Main results}

\subsubsection{Kurihara's analytic quantities}

Let $E/\Q$ be an elliptic curve of conductor $N$ without complex multiplication, and fix an odd prime $p$ such that
\begin{equation}\label{eq:irred}
\tag{sur}
\textrm{$\bar{\rho}: G_{\Q}=\operatorname{Gal}(\overline{\Q}/\Q)\to{\rm Aut}_{\mathbb{F}_p}(E[p])$ is surjective}.
\end{equation}
Denote by $\mathcal{L}$ the set of primes
\[
\mathcal{L}:=\{\ell\;\colon\;\textrm{$\ell\equiv 1\;({\rm mod}\,p)$,\; $\ell\nmid N$,\; and\; $a_\ell\equiv \ell+1\;({\rm mod}\,p)$}\},
\]
where $a_\ell:=\ell+1-\#\widetilde{E}(\mathbb{F}_\ell)$, and let $\mathcal{N}$ be the collection of all squarefree products of primes $\ell\in\mathcal{L}$, with the convention that $1\in\mathcal{N}$ (corresponding to the empty product). For each $\ell\in\mathcal{L}$, let $I_\ell\subset\Z$ be defined by
\[
I_\ell:=(\ell-1,a_\ell-\ell-1)\Z_p\cap\Z,
\]
and for each $n\in\mathcal{N}$ put $I_n=\sum_{\ell\mid n}I_\ell$ if $n\neq 1$ and $I_n=0$ otherwise.

Let $f\in S_2(\Gamma_0(N))$ be the newform attached to $E$. Let  $\Omega_E^+=\int_{E(\R)}\omega_E\in\R$ be the positive N\'eron period of $E$, where $\omega_E$ is a N\'{e}ron differential. Fix a modular parametrization
\[
\phi:X_0(N)\rightarrow E
\]
and let $c_\phi\in\Z$ be the associated \emph{Manin constant}, so that $\phi^*(\omega_E)=c_\phi\cdot 2\pi if(z)dz$. Assume that 
\begin{equation}\label{eq:manin}
p\nmid c_\phi. \tag{$p\nmid c_\phi$}
\end{equation}
(For instance, if $E$ is the strong Weil curve in the $\Q$-isogeny class attached to $f$ and $\phi$ is the optimal quotient, then \eqref{eq:manin} holds provided $p^2\nmid N$ by \cite[Cor.~4.1]{mazur}.) 

Using modular symbols, 
Kurihara  introduced in \cite{kurihara-munster}  certain quantities 
\[
\delta_n\in\Z/I_n
\]
shown to completely determine the structure of the $p$-primary Selmer group ${\rm Sel}_{p^\infty}(E/\Q)$ under some hypotheses   
(see \cite[Thm.~B]{kurihara-munster}). Specifically,  
\[
\delta_n:=\sum_{\substack{a=1\\(a,n)=1}}^n\overline{\left[\frac{a}{n}\right]}\Biggl(\prod_{\ell\mid n}{\rm log}_{\mathbb{F}_\ell}(a)\Biggr)\in \Z/I_n,
\]
where $\left[\frac{a}{n}\right]\in\Q$ is the modular symbol
\[
\left[\frac{a}{n}\right]={\rm Re}\biggl(2\pi i\int_{\infty}^{a/n}f(z)dz\biggr)/\Omega_E^+\in\Q,
\]
which is known to be $p$-integral by our assumptions on $p$ 
(see \cite[\S{3}]{stevens}), 
and $\overline{\left[\frac{a}{n}\right]}\in\Z/I_n$ denotes its reduction modulo $I_n$; and 
\[
{\rm log}_{\mathbb{F}_\ell}:\mathbb{F}_\ell^\times\cong\Z/(\ell-1)\Z\rightarrow\Z/I_\ell
\]
is the discrete logarithm defined by a fixed choice of primitive root $\eta_\ell\in\mathbb{F}_\ell^\times$ (so $\mathbb{F}_\ell^\times=\eta_\ell^{\Z}$), with the last arrow given by reduction modulo $I_\ell$. In particular, for $n=1$, $\delta_n$ reduces to the normalized $L$-value 
\begin{equation}\label{eq:delta-1}
\delta_1=L(E,1)/\Omega_E^+
\in\Q.
\end{equation}

\subsubsection{Kurihara's conjectures}\label{subsec:Kurihara}

Put
\[
{\rm Tam}_E:=\prod_{\ell\mid N}c_\ell,
\]
where $c_\ell=[E(\Q_\ell):E_0(\Q_\ell)]$ is the Tamagawa factor of $E$ at $\ell$. Let $\bar{\delta}_n\in\mathbb{F}_p$ denote the reduction of $\delta_n$ modulo $p$. The following conjecture was first  formulated by Kurihara in \cite[\S{1.2}, Conj.~1]{kurihara-iwasawa2012} in the $p$-ordinary case, and in  \cite[\S{1.2}, Conj.~1.1]{kurihara-TATA} for supersingular  $p$. As in \cite{kim}, it naturally extends to primes $p>2$  of bad reduction as long as hypotheses \eqref{eq:irred} and \eqref{eq:manin} both hold, and below we state it in this level of generality. 

\begin{conjintro}
\label{conj:Kurihara}
Let $p$ be an odd prime such that \eqref{eq:irred} and \eqref{eq:manin} both hold. If $p\nmid{\rm Tam}_E$, then  
\begin{center}
there exists $n\in \mathcal{N}$ such that $\bar{\delta}_{n}\neq 0$.
\end{center}
\end{conjintro}

\subsubsection{Refined Kurihara's conjecture}

In this paper we study a strenghtening of Conjecture~\ref{conj:Kurihara} motivated by the aforementioned structure theorem by Kurihara \cite{kurihara-munster}. The same structure theorem was recently proved from a different perspective by C.-H.~Kim \cite{kim}, whose notations we largely follow. 

For each $n\in\mathcal{N}$ define $\mathscr{M}(n)\in\Z_{\geq 0}\cup\{\infty\}$ by $\mathscr{M}(n):=\infty$ if $\delta_n=0$, and by
\[
\mathscr{M}(n):=\max\{\mathscr{M}\geq 0\;\colon\;\delta_n\in p^{\mathscr{M}}\Z/I_n\}
\]
otherwise. Put $\mathscr{M}_{r}:=\min\{\mathscr{M}(n)\;\colon\;n\in\mathcal{N},\;\nu(n)=r\}$, where $\nu(n)$ denotes the number of prime factors of $n$. As shown in \cite[\S{10.5}]{kurihara-munster}, the quantity $\delta_n$ can be obtained as a coefficient of the Mazur--Tate $\theta$-element of conductor $n$ \cite{mazur-tate}, and it follows from the functional equation of the latter that $\delta_n=0$ unless $(-1)^{\nu(n)}=\epsilon$, where $\nu(n)$ where $\epsilon\in\{\pm{1}\}$ is the root number of $E$.

Putting $\mathscr{M}_{r}:=\min\{\mathscr{M}(n)\;\colon\;\nu(n)=r\}$, then one can show that 
$\mathscr{M}_{r}\geq\mathscr{M}_{r+2}\geq 0$ for all $r\geq 0$, and we put
\[
\mathscr{M}_{\infty}(\boldsymbol{\delta}):=\lim_{\substack{r\to\infty\\(-1)^r=\epsilon}}\mathscr{M}_{r},
\]
where $\boldsymbol{\delta}$ denotes the collection $\{\delta_n\}_{n\in\mathcal{N}}$. Thus  $\mathscr{M}_\infty(\boldsymbol{\delta})<\infty$ if and only if $\boldsymbol{\delta}\neq\{0\}$. Suppose 
this is the case, and let $\varrho$ be the `order of vanishing' of $\boldsymbol{\delta}$:
\[
\varrho:=\min\{r\geq 0\;\colon\;\textrm{$\delta_n\neq 0$\; and\; $\nu(n)=r$}\}.
\]
Then \cite[Thm.~1.8]{kim}, as originally proved by Kurihara \cite[Thm.~B]{kurihara-munster} under some hypotheses and from a different perspective, yields the following exact formula for $\sha_{\rm BK}(E[p^\infty]/\Q):={\rm Sel}_{p^\infty}(E/\Q)/{\rm Sel}_{p^\infty}(E/\Q)_{\rm div}$, where the subscript `${\rm div}$' denotes the maximal divisible submodule:
\begin{equation}\label{eq:str-Selmer}
{\rm ord}_p\bigl(\#\sha_{\rm BK}(E[p^\infty]/\Q)\bigr)=\mathscr{M}_\varrho-\mathscr{M}_{\infty}(\boldsymbol{\delta}).\nonumber
\end{equation}
In particular, if $L(E,1)\neq 0$, then $\varrho=0$ and  Kato's work \cite{kato-euler-systems} implies that $\sha_{\rm BK}(E[p^\infty]/\Q)$ is the same as the $p$-primary part $\sha(E/\Q)[p^\infty]$ of the Tate--Shafarevich group of $E$, so combined with \eqref{eq:delta-1} the above exact formula reduces to
\begin{equation}\label{eq:str-Selmer-rank0}
{\rm ord}_p\bigl(\#\sha(E/\Q)[p^\infty]\bigr)={\rm ord}_p\biggl(\frac{L(E,1)}{\Omega_E^+}\biggr)-\mathscr{M}_{\infty}(\boldsymbol{\delta}).\nonumber
\end{equation}
Together with the Birch--Swinnerton-Dyer formula, this motivates the following conjecture (in arbitrary rank)  formulated in \cite[Conj.~1.9]{kim}. 

\begin{conjintro}[Refined Kurihara's  conjecture]\label{conj:refined-Kur}
Let $p>3$ be a prime such that \eqref{eq:irred} and \eqref{eq:manin} both hold. Then
\[
\mathscr{M}_\infty(\boldsymbol{\delta})={\rm ord}_p({\rm Tam}_E).
\]
\end{conjintro}


\subsubsection{Results} 



Our main result towards the refined Kurihara's conjecture is the following. Let $\Q_\infty/\Q$ denote the cyclotomic $\Z_p$-extension.


\begin{thmintro}
\label{thmintro-Kur}
Let $p>3$ be a prime such that \eqref{eq:irred} and \eqref{eq:manin} both hold. Then the following are equivalent:
\begin{itemize}
\item[(i)] $\mathscr{M}_{\infty}(\boldsymbol{\delta})={\rm ord}_p({\rm Tam}_E)$, and hence Conjecture~\ref{conj:refined-Kur} holds.
\item[(ii)] The Iwasawa Main Conjecture~\ref{conj:IMC-det} for $\Q_\infty/\Q$ holds.
\end{itemize}
\end{thmintro}

In particular, in the case where $p\nmid{\rm Tam}_E$ we deduce:

\begin{corintro}\label{corintro:refined-Kur}
Let $p>3$ be a prime such that \eqref{eq:irred} and \eqref{eq:manin} both hold. 
If $p\nmid{\rm Tam}_E$, then the following are equivalent:
\begin{itemize}
\item[(i)] There exists $n\in \mathcal{N}$ such that $\bar{\delta}_{n}\neq 0$, and hence Conjecture~\ref{conj:Kurihara} holds.
\item[(ii)] The Iwasawa Main Conjecture~\ref{conj:IMC-det} for $\Q_\infty/\Q$ holds.
\end{itemize}
\end{corintro}

\begin{proof}
This is clear from Theorem~\ref{thmintro-Kur}, noting that   $\mathscr{M}_\infty(\boldsymbol{\delta})=0$ if and only if there exists $n\in\mathcal{N}$ such that $\delta_n$ is $p$-indivisible.
\end{proof}

The Iwasawa Main Conjecture~\ref{conj:IMC-det} for $\Q_\infty/\Q$ in the above results is a reformulation of 
\cite[Conj.~12.10]{kato-euler-systems} in terms of determinants of arithmetic complexes. 
Thus combined with known results on the latter, we obtain a proof of Kurihara's conjecture and its refinement in many cases.

\begin{thmintro}
[Refined Kurihara's conjecture]\label{corintro:kur-cases}
Let $p>3$ be a prime such that \eqref{eq:irred} and \eqref{eq:manin} both hold. Then Conjecture~\ref{conj:refined-Kur} (and hence also Conjecture~\ref{conj:Kurihara}) 
holds in each of the following cases:
\begin{itemize}
\item[$\circ$] $p$ is good ordinary for $E$. 
\item[$\circ$] $p$ is good supersingular for $E$ and $N$ is squarefree.
\end{itemize}
\end{thmintro}

\begin{proof}
In the ordinary (resp. supersingular) case, the Iwasawa Main Conjecture for $\Q_\infty/\Q$ in the formulation of Mazur--Swinnerton-Dyer  \cite{mazur-towers,M-SwD} (resp. Kobayashi \cite{kobayashi-ss}) is proved in \cite{wan-hilbert,BCS} (resp. \cite{CLW,BSTW}). Each of these main  conjectures is known to be equivalent to \cite[Conj.~12.10]{kato-euler-systems} and hence to the Iwasawa Main Conjecture~\ref{conj:IMC-det} for $\Q_\infty/\Q$, so the result follows from Theorem~\ref{thmintro-Kur}.
\end{proof}

\subsubsection{Kolyvagin conjectures}\label{subsec:Kolyvagin}

In the 
second part of this paper, we prove an analogue of Theorem~\ref{thmintro-Kur} for W.~Zhang's refinement in \cite{zhang-CDM} of Kolyvagin's conjecture. 

For the statements, let $E/\Q$ be an elliptic curve of conductor $N$, and let $K$ satisfying the \emph{Heegner hypothesis}:
\begin{equation}\label{eq:Heeg}
\textrm{every prime $\ell\mid N$ splits in $K$,}\tag{Heeg}
\end{equation} 
and with discriminant $-D_K<0$ such that
\begin{equation}\label{eq:intro-disc}
\textrm{$D_K$ is odd and $D_K\neq 3$.}\tag{disc}
\end{equation}
Let $p$ be an odd prime such that \eqref{eq:irred} and \eqref{eq:manin} holds. By \cite[Lem.~4.3]{gross-durham}, it follows that $E(K[n])[p]=0$ for all $n\geq 1$, where $K[n]$ denotes the ring class field of $K$ of conductor $n$.

Denote by $\mathcal{L}_{\rm Heeg}$ the set of primes
\[
\mathcal{L}_{\rm Heeg}:=\{\ell\;\colon\;\textrm{$\ell$ is inert in $K$,\; $\ell\nmid N$,\; and\; $a_\ell\equiv \ell+1\equiv 0\;({\rm mod}\,p)$}\},
\] 
and let $\mathcal{N}_{\rm Heeg}$ be the collection of all squarefree products of primes $\ell\in\mathcal{L}_{\rm Heeg}$, with $1\in\mathcal{N}_{\rm Heeg}$ by convention.  For each prime $\ell\in\mathcal{L}_{\rm Heeg}$, let $I_\ell\subset\Z$ be defined by
\[
I_\ell:=(\ell+1,a_\ell)\Z_p\cap\Z,
\]
and for each $n\in\mathcal{N}_{\rm Heeg}$ put $I_n=\sum_{\ell\mid n}I_\ell$ if $n>1$ and $I_n=0$ otherwise. 

Let $T=T_pE$ denote the $p$-adic Tate module of $E$. For each $n\in\mathcal{N}_{\rm Heeg}$, from the Kummer images of Heegner points on $E$ of conductor $n$ associated to a fixed modular parametrization $\phi:X_0(N)\rightarrow E$,  the Kolyvagin derivative process yields a construction of classes
\[
\kappa_n^{\rm Heeg}\in\rH^1(K,T/I_nT)
\]
forming a \emph{Kolyvagin system} $\boldsymbol{\kappa}^{\rm Heeg}=\{\kappa_n^{\rm Heeg}\}_{n}\in\mathbf{KS}(T,\mathcal{F},\mathcal{L}_{\rm Heeg})$ in the sense of \cite{howard} for the Selmer structure $\mathcal{F}$ of Theorem~1.6.5 in \emph{op.\,cit.} (see also \cite[\S{1.7}]{howard}). In particular, $\kappa_1\in\rH^1_\Fcal(K,T)=\varprojlim_m{\rm Sel}_{p^m}(E/K)$ is the Kummer image of the Heegner point $y_K\in E(K)$ in the Gross--Zagier formula \cite{grosszagier}.

The classes $\kappa_n^{\rm Heeg}$ land in the \emph{$n$-transverse Selmer group} $\rH^1_{\Fcal(n)}(K,T/I_nT)\subset\rH^1(K,T/I_nT)$ (see \cite[Lem.~1.7.3]{howard} and \S\ref{subsec:Kolyvagin} below for a review), and 
setting $\mathscr{M}(n)$ to be $\infty$ if $\kappa_n^{\rm Heeg}=0$ and the maximum $\mathscr{M}\geq 0$ such that $\kappa_n^{\rm Heeg}\in p^\mathscr{M}\rH^1_{\Fcal(n)}(K,T/I_nT)$ otherwise, one defines $\mathscr{M}_{r}:=\min\{\mathscr{M}(n)\;\colon\;\nu(n)=r\}$. In this case, Kolyvagin showed $\mathscr{M}_r\geq\mathscr{M}_{r+1}$ for all $r\geq 0$, and we put 
$$\mathscr{M}_\infty(\boldsymbol{\kappa}^{\rm Heeg}):=\lim_{r\to\infty}\mathscr{M}_r\in\Z_{\geq 0}\cup\{\infty\}$$
similarly as in $\S\ref{subsec:Kurihara}$. 

In \cite{kolystructure}, Kolyvagin conjectured that $\boldsymbol{\kappa}^{\rm Heeg}\neq\{0\}$ (equivalently,  $\mathscr{M}_\infty(\boldsymbol{\kappa}^{\rm Heeg})<\infty$). The refined version of this conjecture by W.~Zhang \cite{zhang-CDM} further predicts an exact formula for $\mathscr{M}_\infty(\boldsymbol{\kappa}^{\rm Heeg})$ in terms of arithmetic invariants of $E$. As explained in [\emph{op.\,cit.}, \S{4.4}] (see also \cite[\S{1}]{jetchev}), this is motivated by combining in the case that ${\rm ord}_{s=1}L(E/K,s)=1$:
\begin{itemize}
\item The Gross--Zagier formula for $y_K\in E(K)$, whereby the Birch--Swinnerton-Dyer formula for $L'(E/K,1)$ yields the prediction
\[
2\cdot{\rm ord}_p\bigl([E(K):\Z y_K]\bigr)\overset{?}={\rm ord}_p\bigl(\#\sha(E/K)[p^\infty]\big)+2\cdot{\rm ord}_p\bigl({\rm Tam}_E\bigr).
\]
\item Kolyvagin's structure theorem \cite{kolystructure}, 
which in the case ${\rm ord}_{s=1}L(E/K,s)=1$ implies
\begin{align*}
{\rm ord}_p\bigl(\#\sha(E/K)[p^\infty]\bigr)
&=2\cdot\bigl({\rm ord}_p([E(K):\Z y_K])-\mathscr{M}_\infty(\boldsymbol{\kappa}^{\rm Heeg})\bigr).
\end{align*}
\end{itemize}

The following is \cite[Conj.~4.5]{zhang-CDM}, which allows ${\rm ord}_{s=1}L(E/K,s)>1$. 

\begin{conjintro}[Refined Kolyvagin's  conjecture]\label{conj:refined-Kol}
Let $p$ be an odd prime such that \eqref{eq:irred} and \eqref{eq:manin} both hold. Then
\[
\mathscr{M}_\infty(\boldsymbol{\kappa}^{\rm Heeg})={\rm ord}_p({\rm Tam}_E).
\]
\end{conjintro}


\subsubsection{Results} Our main result toward the refined Kolyvagin conjecture is the following.


\begin{thmintro}[Refined Kolyvagin's conjecture]
\label{thmintro-Kol}
Let $p>3$ be a prime such that \eqref{eq:irred} and \eqref{eq:manin} both hold. If $E$ has good ordinary reduction at $p$ and is unramified in $K$, then the following are equivalent:
\begin{itemize}
\item[(i)] $\mathscr{M}_{\infty}(\boldsymbol{\kappa}^{\rm Heeg})={\rm ord}_p({\rm Tam}_E)$. 
\item[(ii)] The anticyclotomic Iwasawa Main Conjecture~\ref{conj:HPMC-det} holds.
\end{itemize}
In particular, if $p$ splits in $K$, then Conjecture~\ref{conj:refined-Kol} holds.
\end{thmintro}



\subsection{Outline of the proofs}\label{subsec:outline}

The proofs of Theorem~\ref{thmintro-Kur} and Theorem~\ref{thmintro-Kol} go along the same lines as the proofs of related results \cite{BCGS}; our novelty here is in the approach to the descent computations in \emph{op.\,cit.}. 

We shall first explain the proof of Theorem \ref{thmintro-Kur}. Letting $\Lambda=\Z_p[[\Gamma]]$ be the cyclotomic Iwasawa algebra over $\Q$, the proof of Theorem~\ref{thmintro-Kur} relies on a study of the $\Lambda$-adic Kolyvagin system 
\[
\boldsymbol{\kappa}_\Lambda^{\rm Kato}=\{{\kappa}_{n,\Lambda}^{\rm Kato}\in\rH^1(\Q,\mathbf{T}/I_n\mathbf{T})\;\colon\;n\in\mathcal{N}\}
\]
derived from Kato's Euler system, where $\mathbf{T}=T\otimes_{\Z_p}\Lambda$ is the cyclotomic deformation of $T$. By Kato's explicit reciprocity law \cite{kato-euler-systems} and Rohrlich's nonvanishing results \cite{rohrlich-cyc}, ${\kappa}^{\rm Kato}_{1,\Lambda}\in\rH^1(\Q,\mathbf{T})$ is not $\Lambda$-torsion, and so letting ${\kappa}_{1,\Lambda}^{\rm Kato}(\alpha)\in\rH^1(\Q,T(\alpha))$ denote the specialization of ${\kappa}_{1,\Lambda}^{\rm Kato}$ at a non-trivial character $\alpha:\Gamma\rightarrow\Z_p^\times$ with $\alpha\equiv 1\;({\rm mod}\,p^m)$, it follows that ${\kappa}_{1,\Lambda}^{\rm Kato}(\alpha)$ is nonzero for $m\gg 0$. Towards the proof of ${\rm (ii)}\Rightarrow{\rm (i)}$ in Theorem~\ref{thmintro-Kur}, we make use of:
\begin{itemize}
\item[(A)] An exact formula, for $m\gg 0$, for the divisibility index 
\[
{\rm ind}_p\bigl({\kappa}_{1,\Lambda}^{\rm Kato}(\alpha)):=\max\{\mathscr{M}\geq 0\;\colon\;{\kappa}_{1,\Lambda}^{\rm Kato}(\alpha)\in p^\mathscr{M}\rH^1(\Q,T(\alpha))\}
\]
deduced from the Iwasawa Main Conjecture~\ref{conj:IMC-det}. The Tamagawa factors for the Cartier dual $T(\alpha)^*={\rm Hom}(T(\alpha),\mu_{p^\infty})$ appear here. (See Corollary~\ref{cor:index-m-Q}.)
\item[(B)] An exact formula for the $\Z_p$-length of the \emph{strict Selmer group} $\rH^1_{\Fcal_{\rm str}}(\Q,T(\alpha)^*)$ (which for $m\gg 0$ such that $\kappa_{1,\Lambda}^{\rm Kato}(\alpha)\neq 0$ can be shown to be finite) in terms of the difference 
\[
{\rm ind}_p\bigl({\kappa}_{1,\Lambda}^{\rm Kato}(\alpha))-\mathscr{M}_\infty(\boldsymbol{\kappa}_\Lambda^{\rm Kato}(\alpha))
\]
deduced from Mazur--Rubin's structure theorem in \cite{mazrub}. (See Theorem~\ref{thm:str-Kato}.)
\end{itemize}
For $m$ with $p^m\in I_n$, the classes ${\kappa}_{n,\Lambda}^{\rm Kato}(\alpha)$ can be compared to the classes $\kappa_n^{\rm Kato}\in\rH^1(\Q,T/I_nT)$ arising directly from Kato's Euler system (i.e., without going up the cyclotomic tower), and building on (A) and (B) we show that the Iwasawa Main Conjecture~\ref{conj:IMC-det} implies, for $m\gg 0$, the equalities
\[
\mathscr{M}_\infty(\boldsymbol{\kappa}^{\rm Kato})=
\mathscr{M}_\infty(\boldsymbol{\kappa}_{\Lambda}^{\rm Kato}(\alpha))={\rm ord}_p({\rm Tam}_E)+t,
\]
where $t={\rm ord}_p(\#E(\Q_p)[p^\infty])$. Together with the `derived' version of Kato's explicit reciprocity law obtained in \cite{kim} extending Kato's relation between $\kappa_n^{\rm Kato}$ and $\delta_n$ for $n=1$ via the dual exponential map to $n>1$ (see Theorem~\ref{thm:der-ERL}), we thus arrive at the implication ${\rm (ii)}\Rightarrow{\rm (i)}$ in Theorem~\ref{thmintro-Kur}. The proof of the converse is similar: The computations that go into the proof of (A) show that the formula for ${\rm ind}_p(\kappa_{1,\Lambda}^{\rm Kato}(\alpha))$ obtained in Corollary~\ref{cor:index-m-Q} for $m\gg 0$ is equivalent to the Iwasawa Main Conjecture~\ref{conj:IMC-det} specialized at $\alpha$, which together with the divisibility coming from an Euler system argument applied to the non-trivial $\boldsymbol{\kappa}_\Lambda^{\rm Kato}(\alpha)$ yields a proof of the main conjecture. 

The proof of the equivalence in Theorem~\ref{thmintro-Kol} is similar: (A) and (B) then correspond to Corollary~\ref{cor:index-m}  and Theorem~\ref{thm:kol-control}, respectively.  

Most of the work in the paper goes into the proof of (A)  and its analogue for Heegner classes. For this, the new idea in this paper is to build on a reformulation of the Iwasawa Main Conjecture in terms of determinants of Selmer complexes introduced in \cite{kataoka-sano}, 
dispensing with the use of $p$-adic $L$-functions  and control theorems in the style of \cite{greenberg-cetraro} in the approach of \cite{BCGS}. The key descent computations are Proposition \ref{prop:descent-Q} and Proposition \ref{prop:descent}.

\subsection{Comparison to previous works}

\subsubsection{On the refined Kurihara's conjecture}

For ordinary primes $p$, and subject to the `non-anomalous' condition 
\begin{equation}\label{eq:nonanom}
E(\Q_p)[p]=0,\tag{nonanom}
\end{equation}
Corollary~\ref{corintro:refined-Kur} was first proved by Kurihara 
assuming the non-degeneracy of the cyclotomic $p$-adic height pairing  (see \cite[Thm.~10.8]{kurihara-munster}), and more recently by Sakamoto and C.-H.~Kim independently (see \cite[Thm.~1.2]{sakamoto-docmath} and \cite[Thm.~1.10]{kim}). Also in the non-anomalous $p$-ordinary case, Theorem~\ref{thmintro-Kur} (and hence also Corollary~\ref{corintro:refined-Kur}) was essentially proved in \cite{BCGS} (see \cite[Rem.~3.3.5]{BCGS}). 

On the other hand, for non-ordinary (and possibly bad reduction) primes $p$,  Theorem~\ref{thmintro-Kur} was independently proved by C.-H.~Kim assuming $p\nmid{\rm Tam}_E$ (see \cite[Thm.~1.10]{kim}), but the case ${\rm ord}_p({\rm Tam}_E)>0$ of Theorem~\ref{thmintro-Kur} was wide open\footnote{See also forthcoming work by Kurihara--Sakamoto \cite{kurihara-sakamoto} for an independent approach to related results.}.

\subsubsection{On the refined Kolyvagin's conjecture} 

Under certain ramification hypotheses that imply $p\nmid {\rm ord}_p({\rm Tam}_E)$, Conjecture~\ref{conj:refined-Kol} was first proved by W.~Zhang \cite{zhang_ind}. 

Restricting to the case $p$ split in $K$, Theorem~\ref{thmintro-Kol} was first proved in \cite{BCGS} through a study of the $p$-adic $L$-function of Bertolini--Darmon--Prasanna \cite{BDP}, its relation with $y_\infty$ via an explicit reciprocity law \cite{cas-hsieh1}, and its associated Iwasawa--Greenberg Main Conjecture (see Theorem~B and Remark~2.2.5 in \cite{BCGS}). 
Our approach does not require the use of $p$-adic $L$-functions, and applies also in the $p$ inert case. 


\subsection{Acknowledgements}
We heartily thank Masato Kurihara for several enlightening conversations related to the topics of this paper, and his comments and corrections on an earlier draft. We also thank Chan-Ho Kim for helpful comments. At different stages during the preparation of this paper, the first author was supported by the NSF grant DMS-2401321, the 2014-2015 AMS Centennial Research Fellowship, and a JSPS short-term Invitational Fellowship for Research in Japan. The second author was supported by JSPS KAKENHI Grant Number 22K13896.

\section{Kurihara conjectures}

\subsection{Preliminaries}

As in Introduction, let $E/\Q$ be an elliptic curve of conductor $N$ without CM, and fix an odd prime $p$ such that \eqref{eq:irred} and \eqref{eq:manin} both hold.

\subsubsection{Kato's Kolyvagin systems}

As in \cite[\S{2.3}]{kim}, we consider Kato's Euler system 
\begin{equation}\label{eq:kato-ES}
\big\{ z_{np^k}^{}\in\rH^1(\Q(\mu_{np^k}),T)\;\colon\;(n,Np)=1,\;k\geq 0\big\}
\end{equation}
following the normalization in \cite[Thm.~6.1]{kataoka}. In particular, by Kato's explicit reciprocity law \cite{kato-euler-systems} the natural image of ${\rm res}_p(z_1^{})$ under the dual exponential map
${\rm exp}_{\omega_E}^*:\rH^1(\Q_p,T)\rightarrow\Q_p$ is given by 
\begin{equation}\label{eq:kato-ERL}
{\rm exp}^*_{\omega_E}({\rm res}_p(z_1^{}))=\frac{L^{\{p\}}(E,1)}{\Omega_E^+},
\end{equation}
where $L^{\{p\}}(E,s)$ is the $L$-function of $E$ with the Euler factor at $p$ removed.

By \cite[Thm.~3.2.4 and $\S{6.2}$]{mazrub}, 
(see also [\emph{op.\,cit.}, $\S{6.2}$]), by the Kolyvagin derivative process, from \eqref{eq:kato-ES} we obtain a \emph{Kolyvagin system}
\[
\boldsymbol{\kappa}^{\rm Kato}=\big\{\kappa_n^{\rm Kato}\in\rH^1_{\Fcal_{\rm rel}(n)}(\Q,T/I_nT)\;\colon\;n\in\mathcal{N}\big\}
\]
in the sense of \cite[\S{3.1}]{mazrub} for the triple $(T,\mathcal{F}_{\rm rel},\mathcal{L})$, where $\Fcal_{\rm rel}$ is the `canonical' Selmer structure of [\emph{op.\,cit.}, Def.~3.2.1], 
and $\mathcal{F}_{\rm rel}(n)$ denotes the modification of $\mathcal{F}_{\rm rel}$ given by
\begin{equation}\label{eq:F(n)}
\rH^1_{\Fcal_{\rm rel}(n)}(\Q_\ell,T/I_nT)=
\begin{cases}
\rH^1_{\rm tr}(\Q_\ell,T/I_nT):=\rH^1(\Q_\ell(\mu_\ell)/\Q_\ell,\rH^0(\Q_\ell(\mu_\ell),T/I_nT))&\textrm{if $\ell\mid n$,}\\[0.2em]
\rH^1_{\Fcal_{\rm rel}}(\Q_\ell,T/I_nT)&\textrm{if $\ell\nmid n$.}
\end{cases}
\end{equation}
In particular, the construction gives 
\[
\kappa_1^{\rm Kato}=z_1\in\rH^1_{\Fcal_{\rm rel}}(\Q,T),
\]
where
\[
\rH^1_{\Fcal_{\rm rel}}(\Q,T):={\rm ker}\biggl\{\rH^1(\Q,T)\rightarrow\bigoplus_{\ell\neq p}\frac{\rH^1(\Q_\ell,T)}{\rH^1_f(\Q_\ell,T)}\biggr\}
\]
with $\rH^1_f(\Q_\ell,T):={\rm ker}\{\rH^1(\Q_\ell,T)\rightarrow\rH^1(\Q_\ell^{\rm ur},T\otimes_{\Z_p}\Q_p)\}$.

Let $\Q_\infty$ be the cyclotomic $\Z_p$-extension of $\Q$, with $k$-th layer $\Q_k$, and put $\Gamma={\rm Gal}(\Q_\infty/\Q)$ and $\Lambda=\Z_p[[\Gamma]]$. For varying $k\geq 0$, the classes $z_{np^k}$ are compatible under the corestriction maps, and hence for every $n\geq 1$ we can let $z_{n,\Lambda}\in\rH^1(\Q(\mu_n),\mathbf{T)}$ be the class defined by the image of $\varprojlim_kz_{np^k}$ under the composite map
\[
\varprojlim_k\rH^1(\Q(\mu_{np^{k}}),T)\rightarrow\varprojlim_k\rH^1(\Q(\mu_n)\Q_k,T)\cong\rH^1(\Q(\mu_n),\mathbf{T}),
\]
where $\mathbf{T}=T\otimes_{\Z_p}\Lambda$, with
the $G_\Q$-action on $\Lambda$ is given by the inverse of the character $G_\Q\twoheadrightarrow\Gamma\hookrightarrow\Lambda^\times$, and the last isomorphism is given by Shapiro's lemma. By \cite[Thm.~5.3.3]{mazrub} (see also [\emph{op.\,cit.}, \S{6.2}]), the Kolyvagin derivative process applied to $\{z_{n,\Lambda}\}_n$ leads to the construction of a \emph{$\Lambda$-adic Kolyvagin system}
\begin{equation}\label{eq:kato-KS-Lambda}
\boldsymbol{\kappa}_\Lambda^{\rm Kato}=\big\{\kappa_{n,\Lambda}^{\rm Kato}\in\rH^1_{\Fcal_\Lambda(n)}(\Q,\mathbf{T}/I_n\mathbf{T})\;\colon\;n\in\mathcal{N}\big\}
\end{equation}
in the sense of \cite[\S{3.1}]{mazrub} for the triple $(\mathbf{T},\mathcal{F}_\Lambda,\mathcal{L}_E)$, where $\Fcal_\Lambda$ is the Selmer structure of  \cite[Def.~5.3.2]{mazrub} (and $\mathcal{F}_\Lambda(n)$ is the modification of $\Fcal_\Lambda$ analogous to \eqref{eq:F(n)}). Moreover, by construction 
\begin{equation}\label{eq:kato-Lambda-1}
\kappa_{1,\Lambda}^{\rm Kato}=z_\infty:=z_{1,\Lambda}\in\rH^1_{\Fcal_\Lambda}(\Q,\mathbf{T}).
\end{equation}
For any character $\alpha:\Gamma\rightarrow\Z_p^\times$, let $\kappa_{n,\Lambda}^{\rm Kato}(\alpha)\in \rH^1(\Q,T(\alpha)/I_nT(\alpha))$ denote the $\alpha$-specialization of $\kappa_{n,\Lambda}^{\rm Kato}$. 

\begin{lemma}\label{lem:congr-Kato}
Suppose $\alpha:\Gamma\rightarrow\Z_p^\times$ is such that $\alpha\equiv 1\;({\rm mod}\,p^m)$ for some $m\geq 1$. Then for all $n\in\mathcal{N}$ with $p^m\in I_n$ the classes $\kappa_{n,\Lambda}^{\rm Kato}(\alpha)$ and $\kappa_n^{\rm Kato}$ have the same image in $\rH^1(\Q,T(\alpha)/p^mT(\alpha))\cong\rH^1(\Q,T/p^mT)$:
\[
\kappa_{n,\Lambda}^{\rm Kato}(\alpha)\equiv \kappa_n^{\rm Kato}\;({\rm mod}\,p^m).
\]
\end{lemma}

\begin{proof}
This is clear from the construction of the Kolyvagin systems in \cite[App.~A]{mazrub}.
\end{proof}

As in \cite{BCGS}, the specialized Kolyvagin system 
\begin{equation}\label{eq:KS-kato-alpha}
\boldsymbol{\kappa}_\Lambda^{\rm Kato}(\alpha)=\{\kappa_{n,\Lambda}^{\rm Kato}(\alpha)\}_n\in\mathbf{KS}(T(\alpha),\mathcal{F}_{\rm rel},\Lcal)
\end{equation}
(see \cite[Cor.~5.3.15]{mazrub} for the last inclusion) will be of use because of the following fundamental result.

\begin{thm}\label{thm:nonzero-Kato-alpha}
Suppose $\alpha:\Gamma\rightarrow\Z_p^\times$ is a non-trivial character with $\alpha\equiv 1\;({\rm mod}\,p^m)$ for some $m\geq 1$. Then 
\[
\textrm{$\kappa_{1,\Lambda}^{\rm Kato}(\alpha)\neq 0$ for $m\gg 0$.}
\]
In other words, for $m\gg 0$ the Kolyvagin system $\boldsymbol{\kappa}_{\Lambda}^{\rm Kato}(\alpha)$ is non-trivial.
\end{thm}

\begin{proof}
In light of \eqref{eq:kato-Lambda-1}, the result follows from Kato's explicit reciprocity law \cite[Thm.~16.6]{kato-euler-systems} and Rohrlich's nonvanishing results \cite{rohrlich-cyc}.
\end{proof}

\subsubsection{Link with Kurihara's $\delta_n$}

The bridge allowing us to translate our results on Kato's derived cohomology classes $\kappa_n^{\rm Kato}$ to statements about Kurihara's analytic quantities  $\delta_n$ 
is the following extension of Kato's explicit reciprocity law \eqref{eq:kato-ERL}. 

\begin{thm}\label{thm:der-ERL}
Suppose $p>3$ is a prime such that \eqref{eq:irred} and \eqref{eq:manin} both hold. 
Reducing the dual exponential map
${\rm exp}_{\omega_E^*}:\rH^1(\Q_p,T)\rightarrow\Z_p$ modulo $I_n$, for $n>1$, induces a well-defined map
\[
\overline{{\rm exp}_{\omega_E}^*}:\rH^1(\Q_p,T/I_nT)\rightarrow\Z_p/I_n\Z_p
\]
with the property that 
\[
\overline{{\rm exp}_{\omega_E}^*}\bigl({\rm res}_p(\kappa_n^{\rm Kato})\bigr)=u\cdot p^t\cdot\delta_n, 
\]
where $u\in(\Z_p/I_n\Z_p)^\times$  and $p^t=\#E(\Q_p)[p^\infty]$.
\end{thm}

\begin{proof}
This is shown in \cite[Thm.~3.11]{kim} building on computations in \cite[\S\S{6-7}]{KKS}.
\end{proof}

\subsubsection{Exact formula for the strict Selmer group}

We shall need the following consequence of Mazur--Rubin's extension of Kolyvagin's structure theorem for Selmer groups \cite{kolystructure} to the context of \cite{mazrub}. 

Let 
$$T(\alpha)^*:={\rm Hom}(T(\alpha),\mu_{p^\infty})=T(\alpha)^\vee(1)$$
denote the dual of $T(\alpha)$ (where $(-)^\vee$ denotes the Pontryagin dual) and let $\Fcal_{\rm str}$ the Selmer structure \emph{dual} to $\Fcal_{\rm rel}$ in the sense of \cite[\S{2.3}]{mazrub}. In the statement below, the divisibility index $\mathscr{M}_\infty(\boldsymbol{\kappa}_\Lambda^{\rm Kato}(\alpha))\in\Z_{\geq 0}\cup\{\infty\}$ is defined just as $\mathscr{M}_\infty(\boldsymbol{\kappa}^{\rm Heeg})$ in $\S\ref{subsec:Kolyvagin}$ (see also \cite[Def.~5.2.11]{mazrub}, where the notation $\partial^{(\infty)}(\boldsymbol{\kappa}^{\rm Kato})$ is used).

\begin{thm}\label{thm:str-Kato}
Suppose $p>3$ is a prime such that \eqref{eq:irred} holds,  and $\alpha:\Gamma\rightarrow\Z_p^\times$ is a character with $\alpha\equiv 1\;({\rm mod}\,p^m)$ for some $m\geq 1$ such that $\kappa_{1,\Lambda}^{\rm Kato}(\alpha)\neq 0$. Then $\rH^1_{\Fcal_{\rm str}}(K,T(\alpha)^*)$ is finite, with 
\[
{\rm length}_{\Z_p}\bigl(\rH^1_{\Fcal_{\rm str}}(K,T(\alpha)^*)\bigr)=
{\rm ind}_p(\kappa_{1,\Lambda}^{\rm Kato}(\alpha))-\mathscr{M}_\infty(\boldsymbol{\kappa}_\Lambda^{\rm Kato}(\alpha)).
\]
\end{thm}

\begin{proof}
This follows from \cite[Thm.~5.2.12]{mazrub}.
\end{proof}

\subsubsection{Rigidity}

We conclude this section with a technical result showing a certain `rigidity' property for the divisibility index $\mathscr{M}_\infty(\boldsymbol{\kappa}^{\rm Kato})$. As in the proof of the analogous result in \cite[Prop.~2.2.1]{BCGS} for the Heegner point Kolyvagin system, the argument essentially goes back to Kolyvagin (cf. \cite[Prop.~5.2]{mccallum}).

For every $m\geq 1$, let $\mathcal{L}^{(m)}\subset\mathcal{L}$ consist of the primes $\ell\in\mathcal{L}$ with $p^m\in I_n$, and write $\mathcal{N}^{(m)}$ for the set of all squarefree products of primes $\ell\in\mathcal{L}^{(m)}$ (with $1\in\mathcal{N}^{(m)}$ as usual), so clearly we have the chain of inclusions
\begin{equation}\label{eq:nested}
\mathcal{N}=\mathcal{N}^{(1)}\supset\cdots\supset\mathcal{N}^{(m)}\supset\mathcal{N}^{(m+1)}\supset\cdots.
\end{equation}
Let $\boldsymbol{\kappa}=\{\kappa_n\}_{n\in\mathcal{N}}$ denote either $\boldsymbol{\kappa}^{\rm Kato}$ or $\boldsymbol{\kappa}_\Lambda^{\rm Kato}(\alpha)$ for some $\alpha:\Gamma\rightarrow\Z_p^\times$, and define $\mathscr{M}_\infty^{(m)}(\boldsymbol{\kappa})\in\Z_{\geq 0}\cup\{\infty\}$ in the same manner as $\mathscr{M}_\infty(\boldsymbol{\kappa}^{\rm Kato})$ and $\mathscr{M}_\infty(\boldsymbol{\kappa}_\Lambda^{\rm Kato}(\alpha))$, with $\mathcal{N}$ replaced by $\mathcal{N}^{(m)}$. 

\begin{prop}\label{prop:mccullen}
Suppose $p>3$ is a prime satisfying \eqref{eq:irred}. Then 
\[
\mathscr{M}_\infty(\boldsymbol{\kappa})=
\mathscr{M}_\infty^{(m)}(\boldsymbol{\kappa})
\]
for every $m\geq 1$.
\end{prop}

\begin{proof} 
The inequality 
$\mathscr{M}_\infty(\boldsymbol{\kappa})\leq
\mathscr{M}_\infty^{(m)}(\boldsymbol{\kappa})$ 
is immediate from \eqref{eq:nested}. Let $r\geq 0$ be such that 
\[
\mathscr{M}_\infty(\boldsymbol{\kappa})=\mathscr{M}_r:=\min\{{\rm ind}_p(\kappa_n^{})\;\colon\;n\in\mathcal{N}\;\textrm{and}\;\nu(n)=r\}, 
\]
and take $n\in\mathcal{N}$ with $\nu(n)=r$ and ${\rm ind}_p(\kappa_n^{})=\mathscr{M}_r$. If $\mathscr{M}_r=\infty$ there is nothing to show, so in the following we assume $\mathscr{M}_r<\infty$. Put
\[
\mu:=\mathscr{M}_r
\]
for the ease of notation, and note that necessarily $n\in\mathcal{N}^{(\mu+1)}$. If $m\leq \mu+1$ then again there is nothing to show (in view of \eqref{eq:nested}), so in the following we also assume $m>\mu+1$. 

Suppose there exists a prime factor $\ell\mid n$ with $\ell\notin\mathcal{L}^{(m)}$. We shall show that there exists $\ell'\in\mathcal{L}^{(m)}$ such that $n':=\ell'n/\ell$ satisfies ${\rm ind}_p(\kappa_{n'}^{})=\mu$. Repeating this process for any prime factors of $n$ in $\mathcal{L}^{(\mu+1)}\smallsetminus\mathcal{L}^{(m)}$ we will thus arrive at $n''\in\mathcal{N}^{(m)}$ with ${\rm ind}_p(\kappa_{n''}^{})=\mu$; this will show $\mathscr{M}_{\infty}^{(m)}(\boldsymbol{\kappa}^{})\leq \mu$, concluding the proof.

In the following we set $\mathcal{F}=\mathcal{F}_{\rm rel}$ to simplify notation. Letting $\kappa_n^{(\mu+1)}\in\rH^1_{\Fcal(n)}(\Q,T/p^{\mu+1}T)$ be the reduction of $\kappa_n^{}$ modulo $p^{\mu+1}$, by assumption $\kappa_n^{(\mu+1)}\neq 0$  and is contained in $p^\mu\rH^1_{\Fcal(n)}(\Q,T/p^{\mu+1}T)$. Hence via the natural identification
\[
p^\mu\rH^1_{\Fcal(n)}(\Q,T/p^{\mu+1}T)\cong \rH^1_{\Fcal(n)}(\Q,\bar{T}),
\]
where $\bar{T}:=T/pT\cong E[p]$ (see \cite[Lem.~3.5.4]{mazrub}), $\kappa_n^{(\mu+1)}$ defines a nonzero class $\bar{\kappa}\in\rH^1_{\Fcal(n)}(\Q,\bar{T})$. By \cite[Prop.~3.6.1]{mazrub}, there exists  $\ell'\in\mathcal{L}^{(m)}$ with $\ell'\nmid n$ such that ${\rm loc}_{\ell'}(\bar{\kappa})\in \rH^1_{\rm ur}(\Q_\ell,\bar{T})$ is nonzero. Since $\mathcal{L}^{(m)}\subset\mathcal{L}^{(\mu+1)}$, the class $\kappa_{n\ell'}^{(\mu+1)}\in\rH^1_{\Fcal(n\ell')}(\Q,T^{(\mu+1)})$ (obtained by reducing $\kappa_{n\ell'}^{}$ modulo $p^{\mu+1}$) is defined, 
and the Kolyvagin system relations (see \cite[Def.~3.1.3]{mazrub}) show that ${\rm loc}_{\ell'}(\kappa_{n\ell'}^{(\mu+1)})$ and ${\rm loc}_{\ell'}(\kappa_{n}^{(\mu+1)})$ have the same order. Since ${\rm loc}_{\ell'}(\bar{\kappa})\neq 0$, it follows that the class $\bar{\kappa}_{\ell'}\in\rH^1_{\Fcal(n\ell')}(\Q,\bar{T})$ defined by $\kappa_{n\ell'}^{(\mu+1)}$ is also nonzero, and so  $\kappa_{n\ell'}^{(\mu+1)}\notin p^{\mu+1}\rH^1_{\Fcal(n\ell')}(\Q,T/p^{\mu+1}T)$. Moreover, since $\nu(n\ell')=r+1$, from the definitions we see that
\[
{\rm ind}_p(\kappa_{n\ell'}^{})\geq \mathscr{M}_{r+1}=\mathscr{M}_r,
\]
and so ${\rm ind}_p(\kappa_{n\ell'}^{})=\mu$;  moreover, ${\rm loc}_{\ell'}(\bar{\kappa}_{\ell'})\in \rH^1_{\rm tr}(\Q_\ell,\bar{T})$ is nonzero by the Kolyvagin system relations. Now for any place $v$ of $\Q$, let 
\[
\langle\,,\,\rangle_v:\rH^1(\Q_v,\bar{T})\times \rH^1(\Q_v,\bar{T})\rightarrow\mathbb{F}_p
\]
denote the local Tate pairing. For $v\nmid n\ell'$ (resp. $v\mid n/\ell$) the classes ${\rm loc}_v(\bar{\kappa})$ and ${\rm loc}_v(\bar{\kappa}_{\ell'})$ are orthogonal under $\langle\,,\,\rangle_v$, since they are both unramified (resp. transverse) at $v$, so together with the global reciprocity theorem of class field theory we see that
\begin{equation}\label{eq:rec-ells}
0=\sum_v\langle{\rm loc}_{v}(\bar{\kappa}),{\rm loc}_v(\bar{\kappa}_{\ell'})\rangle_v=\langle{\rm loc}_{\ell}(\bar{\kappa}),{\rm loc}_{\ell}(\bar{\kappa}_{\ell'})\rangle_{\ell}+
\langle{\rm loc}_{\ell'}(\bar{\kappa}),{\rm loc}_{\ell'}(\bar{\kappa}_{\ell'})\rangle_{\ell'}.
\end{equation}
Since $\langle\,,\rangle_{\ell'}$ induces a non-degenerate pairing 
\[
\rH^1_{\rm ur}(\Q_{\ell'},\bar{T})\times \rH^1_{\rm tr}(\Q_{\ell'},\bar{T})\rightarrow\mathbb{F}_p,
\]
the above shows that $\langle{\rm loc}_{\ell'}(\bar{\kappa}),{\rm loc}_{\ell'}(\bar{\kappa}_{\ell'})\rangle_{\ell'}\neq 0$, which by \eqref{eq:rec-ells} implies that ${\rm loc}_{\ell_0}(\bar{\kappa}_{\ell'})\neq 0$. Put $n':=\ell'n/\ell\in\mathcal{N}^{(\mu+1)}$. By the same argument as above, 
we see that the nonvanishing of ${\rm loc}_{\ell}(\bar{\kappa}_{\ell'})$ implies that the reduction $\kappa_{n'}^{(\mu+1)}$ of $\kappa_{n'}^{}$  modulo $p^{\mu+1}$ satisfies $\kappa_{n'}^{(\mu+1)}\notin p^{\mu+1}\rH^1_{\Fcal(n')}(\Q,T/p^{\mu+1}T)$, and in fact ${\rm ind}_p(\kappa_{n'}^{(\mu+1)})=\mu$, whence the result.
\end{proof}

\subsection{Determinantal Kato's Main Conjecture}\label{subsec:IMC-kato}

Clearly, our hypothesis \eqref{eq:irred} implies that
\begin{equation}\label{eq:Q-tor}
E(\Q)[p]=0,\tag{${\rm tor}_{\Q}$}
\end{equation}
and we note that only the latter will be needed in this section. 

Let $S$ be the set of places of $\Q$ consisting of the infinite place and the primes dividing $Np$. We set $\Z_S:=\Z[1/Np]$. 
%
As is well-known (see e.g. \cite{fukaya-kato-AMS}), the cohomology complex $\mathbf{R}\Gamma(\Z_S,\Tc)$ is a perfect complex of $\Lambda$-modules, acyclic outside  degrees $1$ and $2$. By \cite[Thm.~12.4(1)]{kato-euler-systems}, its second cohomology group $\rH^2(\Z_S,\Tc)$ is $\Lambda$-torsion, and so taking cohomology (and using that $\mathbf{R}\Gamma(\Z_S,\Tc)$ has Euler characteristic $-{\rm rank}_{\Z_p}(T^-)=-1$, for $T^-$ the minus part of $T$ for complex conjugation) we obtain the canonical isomorphism
\begin{equation}\label{eq:iso-Kato}
Q(\Lambda)\otimes_\Lambda{\rm det}_\Lambda^{-1}\mathbf{R}\Gamma(\Z_S,\Tc)\cong
Q(\Lambda)\otimes_\Lambda\rH^1(\Z_S,\Tc),
\end{equation}
where $Q(\Lambda)$ denotes the fraction field of $\Lambda$.

\begin{rmk}
Although it will not be needed in the following, we note that, as shown in \cite[Thm.~2.18]{burns-sano-IMRN}, just assuming \eqref{eq:Q-tor} one can define a canonical $\Lambda$-module homomorphism
\[
{\det}_\Lambda^{-1}\mathbf{R}\Gamma(\Z_S,\Tc)\rightarrow\rH^1(\Z_S,\Tc)
\]
which is injective if and only if $\rH^2(\Z_S,\Tc)$ is $\Lambda$-torsion, and it induces \eqref{eq:iso-Kato} after extension of scalars to $Q(\Lambda)$.
\end{rmk}

As in \cite{greenvats}, for any prime $\ell\neq p$ let $P_\ell(E,\ell^{-s})$, where $P_\ell(E,X)=(1-\alpha_\ell X)(1-\beta_\ell X)$ (with one or both of $\alpha_\ell,\beta_\ell$ possibly zero) be the Euler factor of $L(E,s)$ at $\ell$, and put
\begin{equation}\label{eq:ell-euler}
\mathcal{P}_\ell:=(1-\alpha_\ell\ell^{-1}\gamma_\ell)(1-\beta_\ell\ell^{-1}\gamma_\ell)\in\Lambda,
\end{equation}
where $\gamma_\ell\in\Gamma$ denotes the Frobenius at $\ell$. Define the \emph{$S$-imprimitive} variant of $z_\infty$ by
\begin{equation}\label{eq:S-impr-Kato}
z_\infty^{(S)}:=\mathcal{P}_N\cdot z_\infty\in\rH^1(\Z_S,\Tc),
\end{equation}
where $\mathcal{P}_N=\prod_{\ell\mid N}\mathcal{P}_\ell$. 
In these terms, the Iwasawa Main Conjecture, in the formulation of \cite[Conj.~3.4]{kataoka-sano}, takes the following form.

\begin{conj}[Main Conjecture for Kato's Euler system]\label{conj:IMC-det}
Let $p$ be an odd prime, and assume that \eqref{eq:Q-tor} holds and $\rH^2(\Z_S,\Tc)$ is $\Lambda$-torsion. Then there exists a $\Lambda$-basis
\[
\mathfrak{z}_{\Q_\infty}\in{\rm det}_\Lambda^{-1}\mathbf{R}\Gamma(\Z_S,\Tc)
\]
that maps to $z_{\infty}^{(S)}$ 
under the canonical isomorphism \eqref{eq:iso-Kato}.
\end{conj}

The following result shows that Conjecture~\ref{conj:IMC-det} is equivalent to the Iwasawa Main Conjecture formulated by Kato in \cite[Conj.~12.10]{kato-euler-systems}.

\begin{prop}\label{prop:equiv-cyc}
Conjecture~\ref{conj:IMC-det} holds if and only if
\begin{equation}\label{eq:IMC-kato}
{\rm char}_\Lambda\bigl(\rH^1(\Z_S,\Tc)/\Lambda\cdot z_{\infty}^{(S)}\bigr)={\rm char}_\Lambda\bigl(\rH^2(\Z_S,\Tc)\bigr)
\end{equation}
as ideals in $\Lambda$.
\end{prop}

\begin{proof}
As noted in \cite[Rem. 7.2]{bks-katoI}, this follows from Proposition~2.1.5 in Chapter~I of \cite{kato-lecture}. (See also \cite[Prop.~3.10]{kataoka-sano}.) 
\end{proof}

\begin{rmk}\label{rem:div-int}
In fact, one can show that the ``upper bound'' divisibility ``$\subset$'' in \eqref{eq:IMC-kato} amounts to the claim that the inverse image $\mathfrak{z}_{\Q_\infty}$ of $z_{\infty}^{(S)}$ under \eqref{eq:iso-Kato} is integral, i.e.,  $\Lambda\cdot\mathfrak{z}_{\Q_\infty}\subset{\det}_\Lambda^{-1}\mathbf{R}\Gamma(\Z_S,\Tc)$.
\end{rmk}

\subsection{Descent computations}\label{subsec:descent-kato}

From the Poitou--Tate duality sequence \eqref{eq:PT-1} below, we see that if the condition 
\begin{equation}\label{eq:rank0}
\textrm{$\rH^1_{\Fcal_{\rm str}}(\Q,T(\alpha)^*)$ is finite}\tag{str-fin}
\end{equation}
holds, 
then $\rH^2(\Z_S,T(\alpha))$ is also finite 
and we have a canonical isomorphism
\[
\vartheta:{\rm det}_{\Q_p}^{-1}\mathbf{R}\Gamma(\Z_S,V(\alpha))\cong\Q_p\otimes_{\Z_p}\rH^1(\Z_S,T(\alpha)),
\]
where $V=\Q_p\otimes_{\Z_p}T$.

\begin{prop}\label{prop:descent-Q}
Suppose $\alpha:\Gamma\rightarrow\Z_p^\times$ is such that $\alpha\equiv 1\;({\rm mod}\,p^m)$ for some $m\geq 1$ and  $\kappa_{1,\Lambda}^{\rm Kato}(\alpha)\neq 0$. Then $\rH^1_{}(\Z_S,T(\alpha))$ is $\Z_p$-free of rank one, $\rH^1_{\Fcal_{\rm str}}(\Q,T(\alpha)^*)$ is finite, and for $m\gg 0$ we have
\begin{align*}
\vartheta\bigl({\rm det}_{\Z_p}^{-1}\mathbf{R}\Gamma(\Z_S,T(\alpha))\bigr)&=\Z_p\cdot\Biggl(\prod_{\ell\mid Np}\#E(\Q_\ell)[p^\infty]\Biggr)\cdot\#\rH^1_{\Fcal_{\rm str}}(\Q,T(\alpha)^*)\cdot x,
\end{align*}
where $x$ is any $\Z_p$-basis of $\rH^1(\Z_S,T(\alpha))$.
\end{prop}

\begin{proof}
By \cite[Thm.~5.2.2]{mazrub} applied to \eqref{eq:KS-kato-alpha}, the nonvanishing of $\kappa_{1,\Lambda}^{\rm Kato}(\alpha)$ implies that $\rH_{\Fcal_{\rm rel}}^1(\Q,T(\alpha))\cong\Z_p$ and  $\#\rH^1_{\Fcal_{\rm str}}(\Q,T(\alpha)^*)<\infty$. By Poitou--Tate duality we have the exact sequence
\begin{equation}\label{eq:PT-1}
\begin{aligned}
0\rightarrow\rH^1_{\Fcal_{\rm rel}}(\Q,T(\alpha))&\rightarrow \rH^1(\Z_S,T(\alpha))\rightarrow\bigoplus_{\ell\mid N}
\frac{\rH^1(\Q_\ell,T(\alpha))}{\rH^1_{f}(\Q_\ell,T(\alpha))}\rightarrow \rH^1_{\Fcal_{\rm str}}(\Q,T(\alpha)^*)^\vee\\
&\quad\rightarrow \rH^2(\Z_S,T(\alpha))\rightarrow\bigoplus_{\ell\mid Np}\rH^2(\Q_\ell,T(\alpha))\rightarrow \rH^0(\Q,T(\alpha)^*)^\vee\rightarrow 0.
\end{aligned}
\end{equation}
Noting that $\rH^1_{f}(\Q_\ell,T(\alpha))=\rH^1(\Q_\ell,T(\alpha))_{\rm tor}=\rH^1(\Q_\ell,T(\alpha))$ for any prime $\ell\mid N$, and that $\rH^0(\Q,T(\alpha)^\ast)=0$ by \eqref{eq:Q-tor}, we see that $\rH^1_{\Fcal_{\rm rel}}(\Q,T(\alpha))=\rH^1(\Z_S,T(\alpha))$, concluding all but the last claim in the statement, and we extract the short exact sequence
\[
0\rightarrow \rH^1_{\Fcal_{\rm str}}(\Q,T(\alpha)^*)^\vee\rightarrow \rH^2(\Z_S,T(\alpha))\rightarrow\bigoplus_{\ell\mid Np}\rH^2(\Q_\ell,T(\alpha))\rightarrow 0.
\]
Since $\rH^2(\Q_\ell,T(\alpha))\cong \rH^0(\Q_\ell,T(\alpha)^*)^\vee$ by local duality, and for $m\gg 0$ we have $\#\rH^0(\Q_\ell,T(\alpha)^*)=\# \rH^0(\Q_\ell,T^*)=\#E(\Q_\ell)[p^\infty]$, we obtain a canonical isomorphism
\begin{align*}
{\rm det}_{\Z_p}^{-1}\mathbf{R}\Gamma(\Z_S,T(\alpha))&\cong{\rm det}_{\Z_p}(\rH^1(\Z_S,T(\alpha)))\otimes{\rm det}_{\Z_p}^{-1}(\rH^2(\Z_S,T(\alpha)))\\
&\cong\Biggl(\prod_{\ell\mid Np}\#E(\Q_\ell)[p^\infty]\Biggr)\cdot{\rm det}_{\Z_p}^{-1}\bigl(\rH^1_{\Fcal_{\rm str}}(\Q,T(\alpha)^*)^\vee\bigr)\otimes{\rm det}_{\Z_p}\bigl(\rH^1(\Z_S,T(\alpha))\bigr)\\
&\cong \Biggl(\prod_{\ell\mid Np}\#E(\Q_\ell)[p^\infty]\Biggr)\cdot\#\rH^1_{\Fcal_{\rm str}}(\Q,T(\alpha)^*)\cdot\rH^1(\Z_S,T(\alpha)),
\end{align*}
which concludes the proof.
\end{proof}

\begin{cor}\label{cor:index-m-Q}
Suppose $\alpha:\Gamma\rightarrow\Z_p^\times$ satisfies $\alpha\equiv 1\;({\rm mod}\,p^m)$ for some $m\geq 1$ and is such that $\kappa_{1,\Lambda}^{\rm Kato}(\alpha)\neq 0$. If Conjecture~\ref{conj:IMC-det} holds, then up to a $p$-adic unit: 
\[
\#\bigl(\rH^1_{\Fcal_{\rm rel}}(\Z_S,T(\alpha))/\Z_p\cdot\kappa_{1,\Lambda}^{\rm Kato}(\alpha)\bigr)=\#E(\Q_p)[p^\infty]\cdot\# \rH^1_{\Fcal_{\rm str}}(\Q,T(\alpha)^*)\cdot{\rm Tam}_E
\]
for $m\gg 0$.
\end{cor}

\begin{proof}
Suppose Conjecture~\ref{conj:IMC-det} holds, and let $\mathfrak{z}_{\Q_\infty}(\alpha)\in{\rm det}_{\Z_p}^{-1}\mathbf{R}\Gamma(\Z_S,T(\alpha))$ be the image of $\mathfrak{z}_{\Q_\infty}$ under the isomorphism
\begin{equation}\label{eq:perfect-control-kato}
{\rm det}_\Lambda^{-1}\mathbf{R}\Gamma(\Z_S,\Tc)\otimes_{\Lambda,\alpha}\Z_p\cong{\rm det}_{\Z_p}^{-1}\mathbf{R}\Gamma(\Z_S,T(\alpha))
\end{equation}
coming from \cite[Prop.~1.6.5(3)]{fukaya-kato-AMS}. Let ${\rm Eul}_N(\alpha)=\mathcal{P}_N(\alpha)$ denote the product of the Euler factors of $L(E,\alpha,s)$ at $s=1$ (see \eqref{eq:ell-euler}). Since $p\nmid N$, we have ${\rm Eul}_N(\alpha)\in\Z_{(p)}$ and for $m\gg 0$ we see that 
\begin{equation}\label{eq:Euler}
{\rm ord}_p({\rm Eul}_N(\alpha))={\rm ord}_p\Biggl(\prod_{\ell\mid N}\#\widetilde{E}_{\rm ns}(\mathbb{F}_\ell)\Biggr).
\end{equation}
Then $\mathfrak{z}_{\Q_\infty}(\alpha)$ is a $\Z_p$-basis of ${\rm det}_{\Z_p}^{-1}\mathbf{R}\Gamma(\Z_S,T(\alpha))$, and by construction its image under $\vartheta$ agrees with ${\rm Eul}_N(\alpha)\cdot\kappa_{1,\Lambda}^{\rm Kato}(\alpha)$. 
As shown in the proof of Proposition~\ref{prop:descent-Q}, we have $\rH^1_{\Fcal_{\rm rel}}(\Q,T(\alpha))=\rH^1(\Z_S,T(\alpha))$, and from there and \eqref{eq:Euler} we see that for $m\gg 0$ we have
\[
{\rm ind}_p(\kappa_{1,\Lambda}^{\rm Kato}(\alpha))={\rm ord}_p(\#E(\Q_p)[p^\infty])+\sum_{\ell\mid N}\bigl({\rm ord}_p(\#E(\Q_\ell)[p^\infty])-{\rm ord}_p(\#\widetilde{E}_{\rm ns}(\mathbb{F}_\ell))\bigr)+
{\rm ord}_p(\rH^1_{\Fcal_{\rm str}}(\Q,T(\alpha)^*)),
\]
which clearly implies the result.
\end{proof}

\subsection{Proof of Theorem~\ref{thmintro-Kur}}

Choose a non-trivial character $\alpha:\Gamma\rightarrow\Z_p^\times$ with $\alpha\equiv 1\;({\rm mod}\,p^m)$ for some $m\geq 1$, and suppose $m$ is large enough so that $\kappa_{1,\Lambda}^{\rm Kato}(\alpha)\neq 0$ (see Theorem~\ref{thm:nonzero-Kato-alpha}).

Suppose the Iwasawa Main Conjecture~\ref{conj:IMC-det} holds. Then from Theorem~\ref{thm:str-Kato} and Corollary~\ref{cor:index-m-Q} we have that $\rH^1_{\Fcal_{\rm str}}(K,T(\alpha)^*)$ is finite, with
\begin{align*}
{\rm ord}_p(\#\rH^1_{\Fcal_{\rm str}}(K,T(\alpha)^*))&=
{\rm ind}_p(\kappa_{1,\Lambda}^{\rm Kato}(\alpha))-\mathscr{M}_\infty(\boldsymbol{\kappa}_\Lambda^{\rm Kato}(\alpha))\\
&={\rm ord}_p(\#E(\Q_p)[p^\infty])+{\rm ord}_{p}(\#\rH^1_{\Fcal_{\rm str}}(K,T(\alpha)^*))+{\rm ord}_p({\rm Tam}_E)-\mathscr{M}_\infty(\boldsymbol{\kappa}_\Lambda^{\rm Kato}(\alpha)),
\end{align*}
and so
\begin{equation}\label{eq:M-kato}
\mathscr{M}_\infty(\boldsymbol{\kappa}_{\Lambda}^{\rm Kato}(\alpha))={\rm ord}_p(\#E(\Q_p)[p^\infty])+{\rm ord}_p({\rm Tam}_E).
\end{equation}
Since the chain of equalities
\begin{equation}\label{eq:chain}
\begin{aligned}
\mathscr{M}_\infty(\boldsymbol{\kappa}_\Lambda^{\rm Kato}(\alpha))&=\mathscr{M}_\infty^{(m)}(\boldsymbol{\kappa}_\Lambda^{\rm Kato}(\alpha))\\
&=
\mathscr{M}_\infty^{(m)}(\boldsymbol{\kappa}^{\rm Kato})=\mathscr{M}_\infty(\boldsymbol{\kappa}^{\rm Kato})=\mathscr{M}_\infty(\boldsymbol{\delta})+{\rm ord}_p(\#E(\Q_p)[p^\infty])
\end{aligned}
\end{equation}
follows from Proposition~\ref{prop:mccullen}, Lemma~\ref{lem:congr-Kato}, Proposition~\ref{prop:mccullen}, and Theorem~\ref{thm:der-ERL}, respectively, combining \eqref{eq:M-kato} and \eqref{eq:chain} this gives the implication ${\rm (ii)}\Rightarrow{\rm (i)}$ in Theorem~\ref{thmintro-Kur}. 

Conversely, suppose the equality  $\mathscr{M}_\infty(\boldsymbol{\delta})={\rm ord}_p({\rm Tam}_E)$ holds, so from \eqref{eq:chain} we deduce that \eqref{eq:M-kato} holds, which together with Theorem~\ref{thm:str-Kato} gives
\begin{equation}\label{eq:div-from-Kur}
{\rm ind}_p(\kappa_{1,\Lambda}^{\rm Kato}(\alpha))={\rm ord}_{p}(\#\rH^1_{\Fcal_{\rm str}}(K,T(\alpha)^*))+{\rm ord}_p(\#E(\Q_p)[p^\infty])+{\rm ord}_p({\rm Tam}_E).
\end{equation}
Let $\mathfrak{z}_{\Q_\infty}\in Q(\Lambda)\otimes_\Lambda{\rm det}_\Lambda^{-1}\mathbf{R}\Gamma(\Z_S,\Tc)$ be the inverse image of $z_\infty^{(S)}\in\rH^1(\Z_S,\Tc)$ under the isomorphism \eqref{eq:iso-Kato}. Since $z_\infty=\kappa_{1,\Lambda}^{\rm Kato}$ is non-torsion by Theorem~\ref{thm:nonzero-Kato-alpha}, by \cite[Thm.~5.3.10]{mazrub} applied to the Kolyvagin system \eqref{eq:kato-KS-Lambda} we deduce the divisibility 
\begin{equation}\label{eq:upper-div}
{\rm char}_\Lambda\bigl(\rH^1(\Z_S,\Tc)/\Lambda\cdot z_{\infty}^{(S)}\bigr)\subset{\rm char}_\Lambda\bigl(\rH^2(\Z_S,\Tc)\bigr),
\end{equation}
and so by Remark~\ref{rem:div-int} we deduce the integrality $\mathfrak{z}_{\Q_\infty}\in{\rm det}^{-1}_\Lambda\mathbf{R}\Gamma(\Z_S,\Tc)$. Letting $\mathfrak{z}_{\Q_\infty}(\alpha)$ be the image of $\mathfrak{z}_{\Q_\infty}$ under the isomorphism \eqref{eq:perfect-control-kato}, the computations in $\S\ref{subsec:descent-kato}$ show that $\mathfrak{z}_{\Q_\infty}(\alpha)$ is a $\Z_p$-basis of ${\rm det}_{\Z_p}^{-1}\mathbf{R}\Gamma(\Z_S,T(\alpha))$ if and only if \eqref{eq:div-from-Kur} holds. Thus from \eqref{eq:div-from-Kur}, \eqref{eq:upper-div},    \cite[Lem.~3.2]{skinner-urban}, and our choice of $\alpha$,  it follows that $\mathfrak{z}_{\Q_\infty}$ is a $\Lambda$-basis of ${\rm det}_{\Z_p}^{-1}\mathbf{R}\Gamma(\Z_S,\Tc)$, thereby concluding the proof.

\section{Kolyvagin conjectures}

\subsection{Preliminaries}

Let $E/\Q$ be an elliptic curve of conductor $N$ and $K$ a quadratic imaginary field satisfying \eqref{eq:Heeg} and \eqref{eq:intro-disc}. Fix an odd prime $p$ such that \eqref{eq:irred} and \eqref{eq:manin} both hold.

\subsubsection{Heegner point Kolyvagin systems}\label{subsec:HPKS}

Keeping the notations from $\S\ref{subsec:Kolyvagin}$, for every integer $m=np^k$ with $k\geq 0$ and $n\in\mathcal{N}_{\rm Heeg}$, let $P[m]\in E(K[m])$ be the Heegner point of conductor $m$ constructed in \cite[\S{1.7}]{howard}, associated with our fixed modular parametrization $\phi:X_0(N)\rightarrow E$. Let $x_m\in\rH^1(K[m],T)$ denote the image of $P[m]$ under the Kummer map $E(K[m])\otimes_{\Z}\Z_p\rightarrow\rH^1(K[m],T)$. By Theorem~1.7.3 in \emph{loc.\,cit.}, after a slight modification the Kolyvagin derivatives of these classes give rise to a Kolyvagin system
\[
\boldsymbol{\kappa}^{\rm Heeg}=\big\{\kappa_n^{\rm Heeg}\in\rH^1_{\Fcal(n)}(K,T/I_nT)\;\colon\;n\in\mathcal{N}_{\rm Heeg}\big\}
\]
in the sense of \cite[\S{1.2}]{howard} for the triple $(T,\mathcal{F},\mathcal{L})$, where $\Fcal$ is the Selmer structure given by the propagation (in the sense of \cite{mazrub}) of the \emph{finite local condition}
\[
\rH^1_f(K_v,T):={\rm ker}\{\rH^1(K_v,T)\rightarrow\rH^1(K_v^{\rm ur},T\otimes\Q_p)\}
\]
for $v\nmid p$, and of the image of the local Kummer map $E(K_v)\otimes\Q_p\rightarrow\rH^1(K_v,T\otimes\Q_p)$ for $v\mid p$; and $\mathcal{F}(n)$ is the modification of $\mathcal{F}$ given by
\begin{equation}\label{eq:F(n)}
\rH^1_{\Fcal(n)}(K_v,T/I_nT)=
\begin{cases}
\rH^1_{\rm tr}(K_v,T/I_nT):=\rH^1(K_v[\ell]/K_v,\rH^0(K_v[\ell],T/I_nT))&\textrm{if $v\mid\ell\mid n$,}\\[0.2em]
\rH^1_{\Fcal}(K_v,T/I_nT)&\textrm{if $v\nmid n$,}
\end{cases}
\end{equation}
where $K_v[\ell]$ denotes the maximal $p$-extension of $K_v$ inside the completion of the ring class field $K[\ell]$ at any of the primes above $v$. In particular, the construction gives
\[
\kappa_1^{\rm Heeg}={\rm Cor}^{K[1]}_K(x_1)\in\rH^1_\Fcal(K,T)\cong\varprojlim_m{\rm Sel}_{p^m}(E/K).
\]

Assume now that 
\begin{equation}\label{eq:ord}
\textrm{$E$ has good ordinary reduction at $p$}\tag{ord}
\end{equation}
and
\begin{equation}\label{eq:unr}
\textrm{$p$ is unramified in $K$.}\tag{unr}
\end{equation}
Let $K_\infty=\cup_{k\geq 0}K_k$ 
be the anticyclotomic $\Z_p$-extension of $K$, and put $\Gamma^{\rm ac}={\rm Gal}(K_\infty/K)$ and $\Lambda^{\rm ac}=\Z_p[[\Gamma^{\rm ac}]]$. Let now $\Tac=T\otimes_{\Z_p}\Lambda^{\rm ac}$ denote the anticyclotomic deformation of $T$, where the $G_K$-action on $\Lambda^{\rm ac}$ is given by the inverse of the character $G_K\twoheadrightarrow\Gamma^{\rm ac}\hookrightarrow(\Lambda^{\rm ac})^\times$. As shown in \cite[\S{2.3}]{howard} and Theorem~4.1.1 
in \cite{eisenstein}, from the Kummer images $x_{np^k}$  for varying $k$ one obtains the construction of a $\Lambda^{\rm ac}$-adic Kolyvagin system
\[
\boldsymbol{\kappa}_\Lambda^{\rm Heeg}=\big\{\kappa_{n,\Lambda}^{\rm Heeg}\in\rH^1_{\Fcal_\Lambda(n)}(K,\Tac/I_n\Tac)\;\colon\;n\in\mathcal{N}_{\rm Heeg}\big\}
\]
for the triple $(\Tac,\Fcal_\Lambda,\Lcal_{\rm Heeg})$, where $\F_\Lambda$ is the Selmer structure  given by
\[
\rH^1_{\Fcal_\Lambda}(K_v,\Tac)=
\begin{cases}
\rH^1_{\rm ur}(K_v,\Tac):=\rH^1(K_v^{\rm ur}/K_v,\rH^0(K_v^{\rm ur},\Tac))&\textrm{if $v\nmid p$,}\\[0.2em]
{\rm im}\{\rH^1(K_v,\Tac_v^+)\rightarrow\rH^1(K_v,\Tac)\}&\textrm{if $v\mid p$,}
\end{cases}
\]
where $\Tac_v^+:=T_v^+\otimes_{\Z_p}\Lambda^{\rm ac}\subset\Tac$ with $T_v^+={\rm ker}(T\rightarrow T_p\widetilde{E})$ the kernel of the reduction map at $v$. Note that we have $\rH^1_{\rm ur}(K_v,\Tac)=\rH^1(K_v,\Tac)$ for $v \nmid p$ (see \cite[Lem. 5.3.1(ii)]{mazrub}).

In particular (see \cite[Rem.~4.1.3]{eisenstein}), letting $\alpha_p$ be the $p$-adic unit root of $x^2-a_px+p$ and $P[p^k]_{\alpha_p}\in E(K[p^k])\otimes\Z_p$ be the \emph{$\alpha_p$-stabilized} Heegner point
\begin{equation}\label{eq:reg-HP}
P[p^k]_{\alpha_p}:=
\begin{cases}
P[p^k]-\alpha_p^{-1}P[p^{k-1}]&\textrm{if $k\geq 1$,}\\[0.2em]
u_K^{-1}(1-\alpha_p^{-1}\sigma_p)(1-\alpha_p^{-1}\sigma_p^*)P[1]&\textrm{if $k=0$ and $p$ splits in $K$,}\\[0.2em]
u_K^{-1}(1-\alpha_p^{-2})P[1]&\textrm{if $k=0$ and $p$ is inert in $K$,}
\end{cases}
\end{equation}
where $u_K=\#(\mathcal{O}_K^\times/\{\pm{1}\})$ (which is $1$ under \eqref{eq:intro-disc}), and $\sigma_p,\sigma_p^*\in{\rm Gal}(K[1]/K)$ denote the Frobenius elements at the primes above $p$, 
the class $\kappa_{1,\Lambda}^{\rm Heeg}\in\rH^1_{\Fcal_\Lambda}(K,\mathbf{T})$
agrees up to a $p$-adic unit with the inverse image 
\begin{equation}\label{eq:y-infty}
y_\infty=\varprojlim_ky_k\in\varprojlim_k\rH^1(K_k,T),
\end{equation}
where $y_k$ 
is the Kummer image of  
$\alpha_p^{-d(k)}{\rm Norm}^{K[p^{d(k)}]}_{K_k}(P[p^{d(k)}]_{\alpha_p})\in E(K_k)\otimes\Z_p$, 
with $d(k)$ the smallest integer  such that $K_k\subset K[p^{d(k)}]$. Note that \eqref{eq:reg-HP} gives rise to
\[
y_0=\begin{cases}
u_K^{-1}
(1-\alpha_p^{-1})^2x_1&\textrm{if $p$ splits in $K$,}\\[0.2em]
u_K^{-1}
(1-\alpha_p^{-2})x_1&\textrm{if $p$ is inert in $K$.}
\end{cases}
\]
A similar computation directly from the construction yields the following result, where we let
\begin{equation}\label{eq:sp-HPKS}
\boldsymbol{\kappa}_\Lambda^{\rm Heeg}(\alpha)=
\{\kappa_{n,\Lambda}^{\rm Heeg}(\alpha)\}_n\in\mathbf{KS}(T(\alpha),\mathcal{F}_\alpha,\Lcal_{\rm Heeg})
\end{equation}
be the Kolyvagin system for $T(\alpha)$ obtained from $\boldsymbol{\kappa}_\Lambda^{\rm Heeg}$ by specialization at $\alpha$, with the Selmer structure $\mathcal{F}_\alpha$ in \cite[Def.~2.1.2]{howard} given by propagating 
\begin{equation}\label{eq:def-Falpha}
\rH^1_{\Fcal_\alpha}(K_v,V(\alpha))=
\begin{cases}
\rH^1_{\rm ur}(K_v,V(\alpha)):=\rH^1(K_v^{\rm ur}/K_v,\rH^0(K_v^{\rm ur},V(\alpha)))&\textrm{if $v\nmid p$,}\\[0.2em]
{\rm im}\{\rH^1(K_v,V_v^+(\alpha))\rightarrow\rH^1(K_v,V(\alpha))\}&\textrm{if $v\mid p$,}
\end{cases}
\end{equation}
where $V_v^+(\alpha)=T_v^+\otimes_{\Z_p}\Q_p(\alpha)\subset V(\alpha)$
(see Remark~1.2.4 and Lemma~2.2.7 in \emph{op.\,cit.} for the inclusion \eqref{eq:sp-HPKS}).

\begin{lemma}\label{lemmacongruence}
Suppose $\alpha:\Gamma^{\rm ac}\rightarrow\Z_p^\times$ satisfies $\alpha\equiv 1\;({\rm mod}\,p^m)$ for some $m\geq 1$. Then for all $n\in\mathcal{N}_{\rm Heeg}$ with $p^m\in I_n$ we have the congruence modulo $p^m$:
\[
\kappa_{n,\Lambda}^{\rm Heeg}(\alpha)\equiv 
\begin{cases} (\alpha_p-1)^2(\beta_p-1)^2 \kappa_n^{\rm Heeg} & \text{if $p$ splits in $K$,} \\[0.2em] 
((p+1)^2-a_p^2) \kappa_n^{\rm Heeg} & \text{if $p$ is inert in $K$,}\end{cases} 
\]
where $\alpha_p,\beta_p$ are the roots of  
$x^2-a_px+p$, with $\alpha_p$ the $p$-adic unit root.
\end{lemma} 

\begin{proof}
See \cite[Lem.~1.1.5]{BCGS}.
\end{proof}


\subsubsection{Exact formula for the $p$-primary Tate--Shafarevich group} 
 
For $\alpha=1$, the next result follows from Kolyvagin's structure theorem \cite[Thm.~C]{kolyvagin-sha} for $\sha(E/\Q)[p^\infty]$; the general case was proved in \cite{BCGS} building on the refinement of Kolyvagin's methods in \cite{howard,zanarella}  and \cite{eisenstein,eisenstein_cyc}. 



\begin{thm}\label{thm:kol-control}  
Suppose $p>3$ is a prime such that \eqref{eq:irred} holds,  and $\alpha:\Gamma^{\rm ac}\rightarrow\Z_p^\times$ is a character such that $\kappa_{1,\Lambda}^{\rm Heeg}(\alpha)\neq 0$. Then $\rH^1_{\Fcal_\alpha}(K,T(\alpha))$ is $\Z_p$-free of rank $1$ and there is a non-canonical isomorphism
\[
\rH^1_{\Fcal_\alpha^*}(K,T(\alpha)^*)\cong(\Q_p/\Z_p)\oplus S_\alpha\oplus S_\alpha
\]
for a finite $\Z_p$-module $S_\alpha$ with
\[
{\rm length}_{\Z_p}(S_\alpha)={\rm ind}_p(\kappa_{1,\Lambda}^{\rm Heeg}(\alpha))-
\mathscr{M}_\infty(\boldsymbol{\kappa}_\Lambda^{\rm Heeg}(\alpha)).
\]
\end{thm}

\begin{proof}
This follows from \cite[Thm.~2.2.2]{BCGS} applied to the Kolyvagin system $\boldsymbol{\kappa}_\Lambda^{\rm Heeg}(\alpha)$.
\end{proof}


\subsection{Determinantal Perrin-Riou's Main Conjecture}\label{subsec:IMC-heegner}

In addition to our running hypotheses \eqref{eq:Heeg} and \eqref{eq:intro-disc},  only condition  
\begin{equation}\label{eq:K-tor}
E(K)[p]=0,\tag{${\rm tor}_K$}
\end{equation}
rather than the stronger \eqref{eq:irred}, 
is needed in this section. We also assume hypotheses \eqref{eq:ord} and \eqref{eq:unr}. 

Let $S$ be the set of places of $K$ consisting of the infinite place and the primes dividing $Np$.  
%
%
Put
\[
A=T\otimes_{\Z_p}\Q_p/\Z_p\cong E[p^\infty],\quad\quad
\mathbf{A}=T\otimes_{\Z_p}(\Lambda^{\rm ac})^\vee\cong{\rm Hom}_{\Z_p}(\mathbf{T},\mu_{p^\infty}),
\]
where $(\Lambda^{\rm ac})^\vee={\rm Hom}_{\Z_p}(\Lambda^{\rm ac},\Q_p/\Z_p)$ denotes the Pontryagin dual of $\Lambda^{\rm ac}$. 

For $v$ a prime of $K$ above $p$, we define $X_v^+\subset X$ for $X\in\{A,\mathbf{A}\}$ using $T_v^+$ in the obvious manner. Recall that the \emph{Selmer complex}  $\widetilde{\mathbf{R}\Gamma}_f(X)$ for $X\in\{T,A,\Tac,\mathbf{A}\}$ is defined by the exact triangle
\begin{equation}\label{eq:exact-tri}
\widetilde{\mathbf{R}\Gamma}_f(X)\rightarrow\mathbf{R}\Gamma(\mathcal{O}_{K,S},X)\rightarrow\bigoplus_{v\mid p}\mathbf{R}\Gamma(K_v,X/X_v^+)\oplus\bigoplus_{v\in S, v\nmid p\infty}\mathbf{R}\Gamma_{/{\rm ur}}(K_v,X),
\end{equation}
where for the primes $v\mid N$ we put
\begin{align*}
\mathbf{R}\Gamma_{/{\rm ur}}(K_v,X):={\rm Cone}\bigl(\mathbf{R}\Gamma(K_v^{\rm ur}/K_v,X^{I_v})\rightarrow\mathbf{R}\Gamma(K_v,X)\bigr)
\end{align*}
(see \cite[(6.1.3.2), (8.8.5)]{nekovar-310}). Put $\widetilde{\rH}_f^i(X):=\rH^i(\widetilde{\mathbf{R}\Gamma}_f(X))$. 

For $X\in\{T,A,\Tac,\mathbf{A}\}$, we also let ${\rm Sel}_{\rm Gr}(X)$ denote the \emph{strict Greenberg Selmer group} 
\begin{equation}\label{eq:str-Gr-T}
{\rm Sel}_{\rm Gr}(X):={\rm ker}\Biggl\{\rH^1(\mathcal{O}_{K,S},X)\rightarrow\bigoplus_{v\mid p}\rH^1(K_v,X/X_v^+)\oplus\bigoplus_{v\in S, v\nmid p\infty}\rH^1(K_v^{\rm ur},X)\Biggr\},
\end{equation}
and note that ${\rm Sel}_{\rm Gr}(\Tc)=\rH^1_{\Fcal_\Lambda}(K,\Tc)$ by definition.

\begin{thm}\label{thm:rank-1}
The module ${\rm Sel}_{\rm Gr}(\Tac)$ is torsion-free and of $\Lambda^{\rm ac}$-rank one.
\end{thm}

\begin{proof}
This is shown in \cite[Thm.~B]{howard} under a certain big image hypothesis on $T$; just assuming \eqref{eq:K-tor}, it follows from the extension of that result in \cite[Thm.~3.4.1]{eisenstein} and \cite[Thm.~6.5.2]{eisenstein_cyc}.
\end{proof}

As explained in \cite[\S{5.2.2}]{kataoka-sano}, using Nekov{\'a}{\v{r}}'s duality \cite[(8.9.6.2)]{nekovar-310}   one deduces from Theorem~\ref{thm:rank-1} the existence of a canonical isomorphism

\begin{equation}\label{eq:duality}
\begin{aligned}
Q(\Lambda^{\rm ac})\otimes_{\Lambda^{\rm ac}}{\rm det}_{\Lambda^{\rm ac}}^{-1}\widetilde{\mathbf{R}\Gamma}_f(\Tac)
&\cong Q(\Lambda^{\rm ac})\otimes_{\Lambda^{\rm ac}}{\rm Sel}_{\rm Gr}(\Tac)\otimes_{\Lambda^{\rm ac}}{\rm Sel}_{\rm Gr}(\Tac)^\iota,
\end{aligned}
\end{equation}
where ${\rm Sel}_{\rm Gr}(\Tac)^\iota:={\rm Sel}_{\rm Gr}(\Tac)\otimes_{\Lambda^{\rm ac}, \iota} {\Lambda^{\rm ac}}$ with $\iota:\Lambda^{\rm ac}\rightarrow\Lambda^{\rm ac}$ denote the involution given by $\gamma\mapsto\gamma^{-1}$ for $\gamma\in\Gamma^{\rm ac}$. 
Let
\[
y_\infty\otimes y_\infty\in{\rm Sel}_{\rm Gr}(\Tac)\otimes_{\Lambda^{\rm ac}}{\rm Sel}_{\rm Gr}(\Tac)^\iota
\]
be defined by the $\Lambda$-adic Heegner class $y_\infty\in{\rm Sel}_{\rm Gr}(\Tac)$ in \eqref{eq:y-infty}. 

\begin{conj}[Main Conjecture for Heegner points]\label{conj:HPMC-det}
Let
\[
\widetilde{\mathfrak{z}}_{K_\infty}\in Q(\Lambda^{\rm ac})\otimes_{\Lambda^{\rm ac}}{\rm det}_{\Lambda^{\rm ac}}^{-1}\widetilde{\mathbf{R}\Gamma}_f(\Tac)
\]
be the element mapping to $y_\infty\otimes y_\infty$ under \eqref{eq:duality}. Then $\widetilde{\mathfrak{z}}_{K_\infty}$  is  a $\Lambda^{\rm ac}$-basis of ${\rm det}_{\Lambda^{\rm ac}}^{-1}\widetilde{\mathbf{R}\Gamma}_f(\Tac)$.
\end{conj}

The following result shows that Conjecture~\ref{conj:HPMC-det} is equivalent to the Iwasawa Main Conjecture for Heegner points formulated by Perrin-Riou in \cite{perrinriou}. Put $X_{\rm Gr}(\mathbf{A})={\rm Sel}_{\rm Gr}(\mathbf{A})^\vee$.

\begin{prop}\label{prop:equiv-ac}
Conjecture~\ref{conj:HPMC-det} holds if and only if
\begin{equation}\label{eq:Heeg-IMC}
{\rm char}_{\Lambda^{\rm ac}}\bigl({\rm Sel}_{\rm Gr}(\Tac)/\Lambda^{\rm ac}\cdot y_\infty\bigr)\cdot
{\rm char}_{\Lambda^{\rm ac}}\bigl({\rm Sel}_{\rm Gr}(\Tac)/\Lambda^{\rm ac}\cdot y_\infty\bigr)^\iota={\rm char}_{\Lambda^{\rm ac}}\bigl(X_{\rm Gr}(\mathbf{A})_{\rm tors}\bigr)
\end{equation}
as ideals in $\Lambda^{\rm ac}$, where the subscript ${\rm tors}$ denotes the $\Lambda^{\rm ac}$-torsion submodule.
\end{prop}

\begin{proof}
In light of Theorem~\ref{thm:rank-1}, this follows from \cite[Prop.~5.12]{kataoka-sano}.
\end{proof}

\begin{rmk}\label{rem:div-det}
As in $\S\ref{subsec:IMC-kato}$, one shows that the upper bound divisibility ``$\subset$'' in \eqref{eq:Heeg-IMC} amounts to the integrality $\Lambda^{\rm ac}\cdot\widetilde{\mathfrak{z}}_{K_\infty}\subset{\det}_{\Lambda^{\rm ac}}^{-1}\widetilde{\mathbf{R}\Gamma}_f(\Tac)$.
\end{rmk}

\subsection{Descent computations}\label{subsec:descent-K}

Put
\begin{equation}\label{eq:def-sha}
\sha_{\rm BK}(T(\alpha)^*/K):=\frac{\rH^1_{\Fcal_\alpha^*}(K,T(\alpha)^*)}{\rH^1_{\Fcal_\alpha^*}(K,T(\alpha)^*)_{\rm div}},\nonumber
\end{equation}
where $\rH^1_{\Fcal_\alpha^*}(K,T(\alpha)^*)_{\rm div}$ denotes the maximal divisible submodule of $\rH^1_{\Fcal_\alpha^*}(K,T(\alpha)^*)$. 

\begin{rem}
For $\alpha=\mathds{1}$, the Selmer condition $\Fcal_\alpha$ is the same as the Selmer condition $\Fcal$ in $\S\ref{subsec:HPKS}$, and one can show (see e.g. \cite{flach-CT}) that $\rH^1_{\Fcal^*}(K,T^*)$ agrees with the \emph{Bloch--Kato Selmer group}
\[
{\rm Sel}_{\rm BK}(K,T^*):={\rm ker}\biggl\{\rH^1(\mathcal{O}_{K,S},T^*)\rightarrow\bigoplus_{v\in S}\frac{\rH^1(K_v,T^*)}{\rH^1_f(K_v,T^*)}\biggr\},
\]
where the local conditions $\rH^1_f(K_v,T^*)$ are obtained by propagating via $V\twoheadrightarrow A\cong T^*$ the \emph{finite subspace}
\[
\rH^1_f(K_v,V):=
\begin{cases}
{\rm ker}\{\rH^1(K_v,V)\rightarrow \rH^1(K_v,V\otimes_{\Q_p}\mathbf{B}_{\rm cris})\}&\textrm{if $v\mid p$},\\[0.2em]
{\rm ker}\{\rH^1(K_v,V)\rightarrow \rH^1(K_v^{\rm ur},V)\}&\textrm{else.}
\end{cases}
\]
Thus in this case $\sha_{\rm BK}(T^*/K)$ 
is the same as Bloch--Kato Tate--Shafarevich group of $T^*$ (and hence equal to the $p$-primary part $\sha(E/K)[p^\infty]$ of the usual Tate--Shafarevich group of $E$ when the latter is finite). However, we note that for general $\alpha$, $\rH^1_{\Fcal_\alpha^*}(K,T(\alpha)^*)$ may differ from the Bloch--Kato Selmer groups attached to $T(\alpha)^*$.
\end{rem}

The next result is a twisted analogue of (3.4.4) in \cite{sano-derived}, whose approach we largely follow.

\begin{prop}\label{prop:descent}
Suppose $\alpha:\Gamma^{\rm ac}\rightarrow\Z_p^\times$ is such that $\alpha\equiv 1\;({\rm mod}\,p^m)$ for some $m\geq 1$ and $\kappa_{1,\Lambda}^{\rm Heeg}(\alpha)\neq 0$. Then $\rH^1_{\Fcal_\alpha}(K,T(\alpha))$ and $\rH^1_{\Fcal_{\alpha^{-1}}}(K,T(\alpha^{-1}))$ are both $\Z_p$-free of rank $1$, and there is a canonical isomorphism
\[
\vartheta:\Q_p\otimes_{\Z_p}{\rm det}_{\Z_p}^{-1}\widetilde{\mathbf{R}\Gamma}_f(T(\alpha))\cong\Q_p\otimes_{\Z_p}\bigl(\rH^1_{\Fcal_\alpha}(K,T(\alpha))\otimes_{\Z_p}\rH^1_{\Fcal_{\alpha^{-1}}}(K,T(\alpha^{-1}))\bigr),
\]
with
\begin{align*}
\vartheta\bigl({\rm det}_{\Z_p}^{-1}\widetilde{\mathbf{R}\Gamma}_f(T(\alpha))\bigr)&=L_p^2\cdot({\rm Tam}_E)^2\cdot\#\sha_{\rm BK}(T(\alpha)^*/K)\cdot\Z_p
\otimes_{\Z_p}\bigl(\rH^1_{\Fcal_\alpha}(K,T(\alpha))\otimes_{\Z_p}\rH^1_{\Fcal_{\alpha^{-1}}}(K,T(\alpha^{-1}))\bigr)
\end{align*}
for $m\gg 0$, where $L_p=\prod_{v\mid p}\#\widetilde{E}(\mathbb{F}_v)$ with $\mathbb{F}_v$ the residue field of $K$ at $v$.
\end{prop}

\begin{proof}
By the action of complex conjugation, the nonvanishing of $\kappa_{1,\Lambda}^{\rm Heeg}(\alpha)$ implies that of $\kappa_{1,\Lambda}^{\rm Heeg}(\alpha^{-1})$, so the first claim follows from \eqref{eq:sp-HPKS}  and 
\cite[Thm.~1.6.1]{howard}.  

For $X\in\{T(\alpha),T(\alpha)^*\}$, the exact triangle \eqref{eq:exact-tri} gives rise to the exact sequence
\[
0\rightarrow\bigoplus_{v\mid p}\rH^0(K_v,X/X_v^+)\rightarrow\widetilde{\rH}^1_f(K,X)\rightarrow {\rm Sel}_{\rm Gr}(K,X)\rightarrow 0,
\]
where, similarly as in \eqref{eq:str-Gr-T}, ${\rm Sel}_{\rm Gr}(K,X)$ is the \emph{strict Greenberg Selmer group} defined by
\[
{\rm Sel}_{\rm Gr}(K,X):={\rm ker}\biggl\{\rH^1(\mathcal{O}_{K,S},X)\rightarrow\bigoplus_{v\mid p}\frac{\rH^1(K_v,X)}{\rH^1(K_v,X_v^+)}\oplus\bigoplus_{v\mid N}\frac{\rH^1(K_v,X)}{\rH^1_{\rm ur}(K_v,X)}\biggr\}.
\]
In particular, since $\rH^0(K_v,A(\alpha)/A^+_v(\alpha))$ is finite for all $v\mid p$, it follows that
\begin{equation}\label{eq:ext-H-1}
\widetilde{\rH}^1_f(K,T(\alpha))={\rm Sel}_{\rm Gr}(K,T(\alpha)),
\end{equation}
and together with the duality $\widetilde{\rH}_f^2(K,T(\alpha))\cong\widetilde{\rH}^1_f(K,T(\alpha)^*)^\vee$ from \cite[Prop.~9.7.2(i)]{nekovar-310}, that $\widetilde{\rH}_f^2(K,T(\alpha))$ fits into the short exact sequence
\begin{equation}\label{eq:ext-H-2}
0\rightarrow{\rm Sel}_{\rm Gr}(K,T(\alpha)^*)^\vee\rightarrow\widetilde{\rH}_f^2(K,T(\alpha))\rightarrow\bigoplus_{v\mid p}\rH^0(K_v,A(\alpha^{-1})/A_v^+(\alpha^{-1}))^\vee\rightarrow 0,
\end{equation}
where we used that $T(\alpha)^*\cong A(\alpha^{-1})$ in writing the last term. 
Putting 
\[
A_v^-(\alpha^{-1}):=A(\alpha^{-1})/A_v^+(\alpha^{-1})
\]
for the ease of notation, from \eqref{eq:ext-H-1} and \eqref{eq:ext-H-2} we thus obtain a canonical isomorphism
\begin{equation}\label{eq:nek-duality}
\begin{aligned}
&{\rm det}_{\Z_p}^{-1}\widetilde{\mathbf{R}\Gamma}_f(K,T(\alpha))\\
&\quad\cong\Biggl(\prod_{v\mid p}\# \rH^0(K_v,A_v^-(\alpha^{-1}))\Biggr)\cdot{\rm det}_{\Z_p}({\rm Sel}_{\rm Gr}(K,T(\alpha)))\otimes_{\Z_p}{\rm det}_{\Z_p}^{-1}({\rm Sel}_{\rm Gr}(K,T(\alpha)^*)^\vee).
\end{aligned}
\end{equation}

On the other hand, from Poitou--Tate duality we have the exact sequence
\begin{equation}\label{eq:PT}
\begin{aligned}
0\rightarrow{\rm Sel}_{\rm Gr}(K,T(\alpha))\rightarrow \rH^1_{\Fcal_\alpha}(K,T(\alpha))&\rightarrow\bigoplus_{v\mid p}\frac{\rH^1_{\Fcal_\alpha}(K_v,T(\alpha))}{\rH^1(K_v,T_v^+(\alpha))}\oplus\bigoplus_{v\mid N}\frac{\rH^1_{\Fcal_\alpha}(K_v,T(\alpha))}{\rH^1_{\rm ur}(K_v,T(\alpha))}\\
&\quad\quad\rightarrow{\rm Sel}_{\rm Gr}(K,T(\alpha)^*)^\vee\rightarrow\rH^1_{\Fcal_\alpha^*}(K,T(\alpha)^*)^\vee\rightarrow 0.
\end{aligned}
\end{equation}
For $v\mid N$, we have $\rH^1_{\rm ur}(K_v,T(\alpha))\subset \rH^1_{\Fcal_\alpha}(K_v,T(\alpha))$ and by definition (see \cite[(1.38)]{burns-flach-motivic}, for example)
\begin{equation}\label{eq:tam}
{\rm Tam}(T(\alpha)^*/K_v)\cdot\Z_p=\#\biggl(\frac{\rH^1_{\Fcal_\alpha}(K_v,T(\alpha))}{\rH^1_{\rm ur}(K_v,T(\alpha))}\biggr)\cdot\Z_p.
\end{equation} 
Taking determinants in the exact sequence \eqref{eq:PT}, and using \eqref{eq:tam} and the short exact sequence
\[
0\rightarrow\sha_{\rm BK}(T(\alpha)^*/K)^\vee\rightarrow\rH^1_{\Fcal_\alpha^*}(K,T(\alpha)^*)^\vee\rightarrow\bigl(\rH^1_{\Fcal_\alpha^*}(K,T(\alpha)^*)_{\rm div}\bigr)^\vee\rightarrow 0,
\]
we obtain a canonical isomorphism
\begin{equation}\label{eq:PT-duality}
\begin{aligned}
{\rm det}_{\Z_p}({\rm Sel}_{\rm Gr}(K,T(\alpha)))&\otimes_{\Z_p}{\rm det}_{\Z_p}^{-1}({\rm Sel}_{\rm Gr}(K,T(\alpha)^*)^\vee)\\
&\quad\cong\#\sha_{\rm BK}(T(\alpha)^*/K)\cdot\Biggl(\prod_{v\mid N}{\rm Tam}(T(\alpha)^*/K_v)\Biggr)\cdot\Biggl(\prod_{v\mid p}\#\biggl(\frac{\rH^1_{\Fcal_\alpha}(K_v,T(\alpha))}{\rH^1(K_v,T_v^+(\alpha))}\biggr)\Biggr)\cdot\Z_p\\
&\quad\quad\otimes_{\Z_p}{\rm det}_{\Z_p}(\rH^1_{\Fcal_\alpha}(K,T(\alpha)))\otimes_{\Z_p}{\rm det}_{\Z_p}^{-1}((\rH^1_{\Fcal_\alpha^*}(K,T(\alpha)^*)_{\rm div})^\vee).
\end{aligned}
\end{equation}

Noting that 
$(\rH^1_{\Fcal_\alpha^*}(K,T(\alpha)^*)_{\rm div})^\vee\cong{\rm Hom}_{\Z_p}(\rH^1_{\Fcal_{\alpha^{-1}}}(K,T(\alpha^{-1})),\Z_p)$, combining \eqref{eq:nek-duality} and \eqref{eq:PT-duality} this shows the existence of a canonical isomorphism $\vartheta$ as in the statement, and using that
\[
\Biggl(\prod_{v\mid N}{\rm Tam}(T(\alpha)^*/K_v)\Biggr)\cdot\Z_p=({\rm Tam}_E)^2\cdot\Z_p
\]
for $m\gg 0$ (see \cite[Rem.~1.2.9]{BCGS}) and Lemma~\ref{lem:ord-BK} below, the result follows.
\end{proof}

The next result is a twisted analogue of \cite[Prop.~2.5]{greenberg-cetraro}.

\begin{lemma}\label{lem:ord-BK}
Suppose $\alpha:\Gamma^{\rm ac}\rightarrow\Z_p^\times$  satisfies $\alpha\equiv 1\;({\rm mod}\,p^m)$ for some $m\gg 0$. Then for every prime $v$ of $K$ above $p$ we have
\[
\# \rH^0(K_v,A_v^-(\alpha^{\pm{1}}))\cdot\Z_p=\#\biggl(\frac{\rH^1_{\Fcal_\alpha}(K_v,T(\alpha^{\pm{1}}))}{\rH^1(K_v,T_v^+(\alpha^{\pm{1}}))}\biggr)\cdot\Z_p=
\#\widetilde{E}(\mathbb{F}_v)\cdot\Z_p
\]
for $m\gg 0$.
\end{lemma}

\begin{proof}
Let $v$ be a prime of $K$ above $p$. 
Since $\rH^1_{\Fcal_\alpha}(K_v,A(\alpha))$ is the image of $\rH^1_{\Fcal_\alpha}(K_v,V(\alpha))$ in \eqref{eq:def-Falpha} under the natural map 
induced by $V(\alpha)\twoheadrightarrow A(\alpha)$, we see that
\begin{equation}\label{eq:f-div}
\rH^1_{\Fcal_\alpha}(K_v,A(\alpha))=\rH^1(K_v,A_v^+(\alpha))_{\rm div}.
\end{equation}
From the short exact sequence $0 \to T_v^+(\alpha) \to V_v^+(\alpha) \to A_v^+(\alpha) \to 0$ we obtain
\[
\frac{\rH^1(K_v,A_v^+(\alpha))}{\rH^1(K_v,A_v^+(\alpha))_{\rm div}}\cong \rH^2(K_v, T_v^+(\alpha)),
\]
which together with \eqref{eq:f-div} and local Tate duality amounts to
\[
\frac{\rH^1_{\Fcal_\alpha}(K_v,T(\alpha^{-1}))}{\rH^1(K_v,T_v^+(\alpha^{-1}))}\cong\rH^0(K_v,A_v^-(\alpha^{-1})),
\]
whence the first equality in the statement. The second equality is immediate from the fact that $A(\alpha^{\pm{1}})[p^m]\cong A[p^m]$ as $G_K$-modules, and so $\rH^0(K_v,A_v^-(\alpha^{\pm{1}}))\cong\rH^0(K_v,A_v^-)\cong \widetilde{E}(\mathbb{F}_v)[p^\infty]$.
%
\end{proof}

Thus we arrive at the following result.

\begin{cor}\label{cor:index-m}
Suppose $\alpha:\Gamma^{\rm ac}\rightarrow\Z_p^\times$ is such that $\alpha\equiv 1\;({\rm mod}\,p^m)$ for some $m\geq 1$ and $\kappa_{1,\Lambda}^{\rm Heeg}(\alpha)\neq 0$. If Conjecture~\ref{conj:HPMC-det} holds, then up to a $p$-adic unit: 
\[
\#\bigl(\rH^1_{\Fcal_\alpha}(K,T(\alpha))/\Z_p\cdot\kappa_{1,\Lambda}^{\rm Heeg}(\alpha)\bigr)^2=\Biggl(
\prod_{v\mid p}\#\widetilde{E}(\mathbb{F}_v)\Biggr)^2\cdot\bigl({\rm Tam}_E\bigr)^2\cdot\#\sha_{\rm BK}(T(\alpha)^*/K)
\]
for $m\gg 0$.
\end{cor}

\begin{proof}
Assume Conjecture~\ref{conj:HPMC-det}, let $\widetilde{\mathfrak{z}}_{K_\infty}$ be the $\Lambda^{\rm ac}$-basis of ${\rm det}_{\Lambda^{\rm ac}}^{-1}\widetilde{\mathbf{R}\Gamma}_f(\Tac)$  corresponding to $y_\infty\otimes y_\infty$ under the isomorphism \eqref{eq:duality}, and let $\widetilde{\mathfrak{z}}_{K_\infty}(\alpha)$ denote 
its image under the surjection
$${\rm det}_{\Lambda^{\rm ac}}^{-1}\widetilde{\mathbf{R}\Gamma}_f(\Tac) \twoheadrightarrow 
{\rm det}_{\Z_p}^{-1}\widetilde{\mathbf{R}\Gamma}_f(T(\alpha))$$
induced by the `perfect control' isomorphism
\begin{equation}\label{eq:deriv-control}
\widetilde{\mathbf{R}\Gamma}_f(\Tac)\otimes_{\Lambda^{\rm ac},\alpha}^{\mathbf{L}}\Z_p\cong
\widetilde{\mathbf{R}\Gamma}_f(T(\alpha))
\end{equation}
of \cite[Prop.~8.10.1]{nekovar-310}. 
Then $\widetilde{\mathfrak{z}}_{K_\infty}(\alpha)$ is a $\Z_p$-basis of ${\rm det}_{\Z_p}^{-1}\widetilde{\mathbf{R}\Gamma}_f(T(\alpha))$ and 
its image under the isomorphism $\vartheta$ of Proposition~\ref{prop:descent} is given by $\kappa_{1,\Lambda}^{\rm Heeg}(\alpha)\otimes\kappa_{1,\Lambda}^{\rm Heeg}(\alpha^{-1})$ up to a $p$-adic unit. 
Noting that
\[
\#\bigl(\rH^1_{\Fcal_\alpha}(K,T(\alpha))/\Z_p\cdot\kappa_{1,\Lambda}^{\rm Heeg}(\alpha)\bigr)=\#\bigl(\rH^1_{\Fcal_{\alpha^{-1}}}(K,T(\alpha^{-1}))/\Z_p\cdot\kappa_{1,\Lambda}^{\rm Heeg}(\alpha^{-1})\bigr)
\]
by the action of complex conjugation, the result thus follows from Proposition~\ref{prop:descent}.
\end{proof}

\begin{rem}
For $p$ split in $K$, Corollary~\ref{cor:index-m} was proved in \cite[Cor.~1.2.11]{BCGS} based on a study of the $p$-adic $L$-function of \cite{BDP}. It is unclear whether the method in \cite{BCGS} can be extended to the nonsplit case.
\end{rem}

\subsection{Proof of Theorem~\ref{thmintro-Kol}}

With the results in the preceding sections in hand, the proof now follows similarly as in the case of  Theorem~\ref{thmintro-Kur}, so we shall be rather brief. 

Choose a non-trivial $\alpha:\Gamma^{\rm ac}\rightarrow\Z_p^\times$ with $\alpha\equiv 1\;({\rm mod}\,p^m)$ for some $m\geq 1$, and assume $m$ is large enough so that $\kappa_{1,\Lambda}^{\rm Heeg}(\alpha)\neq 0$ (as is possible by the nonvanishing results of Cornut--Vatsal \cite{cornut,vatsal}).

If the anticyclotomic Iwasawa Main Conjecture~\ref{conj:HPMC-det} holds, then  Theorem~\ref{thm:kol-control} and Corollary~\ref{cor:index-m} give
\begin{equation}\label{eq:ind-lambda-Heeg}
\mathscr{M}_\infty(\boldsymbol{\kappa}_\Lambda^{\rm Heeg}(\alpha))=\sum_{v\mid p}{\rm ord}_p(\#\widetilde{E}(\mathbb{F}_v))+{\rm ord}_p({\rm Tam}_E).
\end{equation}
With notations as in Lemma~\ref{lemmacongruence}, we note that
\[
\Biggl(\prod_{v\mid p}\#\widetilde{E}(\mathbb{F}_v)\Bigg)\cdot\Z_p=\begin{cases}
(\alpha_p-1)^2(\beta_p-1)^2\cdot\Z_p&\textrm{if $p$ splits in $K$,}\\[0.2em]
((p+1)^2-a_p^2)\cdot\Z_p&\textrm{if $p$ is inert in $K$.}
\end{cases}
\]
Thus from \cite[Prop.~2.2.1]{BCGS} (the analogue of Proposition~\ref{prop:mccullen} in the present  setting), Lemma~\ref{lemmacongruence}, and \eqref{eq:ind-lambda-Heeg} we arrive at
\[
\mathscr{M}_\infty(\boldsymbol{\kappa}^{\rm Heeg})={\rm ord}_p({\rm Tam}_E),
\]
concluding the proof of the implication ${\rm (ii)}\Rightarrow{\rm (i)}$ in Theorem~\ref{thmintro-Kol}. The proof of ${\rm (i)}\Rightarrow{\rm (ii)}$ follows similarly as in the case of Theorem~\ref{thmintro-Kur}, combining:
\begin{enumerate}
\item the divisibility
$\Lambda^{\rm ac}\cdot\widetilde{\mathfrak{z}}_{K_\infty}\subset{\det}_{\Lambda^{\rm ac}}^{-1}\widetilde{\mathbf{R}\Gamma}_f(\Tac)$ (see Remark~\ref{rem:div-det}) deduced from \cite[Thm.~2.2.10]{howard} and the nonvanishing of $\kappa_{1,\Lambda}^{\rm Heeg}$;
\item the equality $\Z_p\cdot\widetilde{\mathfrak{z}}_{K_\infty}(\alpha)={\rm det}_{\Z_p}^{-1}\widetilde{\mathbf{R}\Gamma}_f(T(\alpha))$ coming from the assumption $\mathscr{M}_\infty(\boldsymbol{\kappa}^{\rm Heeg})={\rm ord}_p({\rm Tam}_E)$ in (i) and the calculations in $\S\ref{subsec:descent-K}$,
\end{enumerate}
and invoking \cite[Lem.~3.2]{skinner-urban} to conclude.

\bibliography{references}
\bibliographystyle{alpha}
\end{document}